\documentclass[a4paper, english, 10pt]{amsart}

\usepackage{savesym}
\usepackage{fourier}
\usepackage{comment}
\usepackage{enumerate}
\numberwithin{equation}{section}
\usepackage{graphicx}
\usepackage[
    backgroundcolor=gray!10,
    linecolor=gray!20,
    bordercolor=gray!20
]{todonotes}
\usepackage{amsfonts,amsmath,amstext,amsbsy,amssymb}
\usepackage{amsopn,amsthm,amscd,amsxtra}
\usepackage{latexsym,mathrsfs}

\usepackage{mathabx}
\savesymbol{widering}
\restoresymbol{widering}{MABwidering}
\newcommand{\carr}{\righttoleftarrow}

\usepackage{hyperref}
\RequirePackage[dvipsnames]{xcolor}
\definecolor{halfgray}{gray}{0.55}
\definecolor{webgreen}{rgb}{0,0.5,0}
\definecolor{webbrown}{rgb}{.6,0,0} \hypersetup{%
  colorlinks=true, linktocpage=true, pdfstartpage=3,
  pdfstartview=FitV,%
  breaklinks=true, pdfpagemode=UseNone, pageanchor=true,
  pdfpagemode=UseOutlines,%
  plainpages=false, bookmarksnumbered, bookmarksopen=true,
  bookmarksopenlevel=1,%
  hypertexnames=true,
  pdfhighlight=/O,
  urlcolor=webbrown, linkcolor=RoyalBlue,
  citecolor=webgreen,
  pdftitle={},%
  pdfauthor={},%
  pdfsubject={2000 MAthematical Subject Classification: Primary:},%
  pdfkeywords={},%
  pdfcreator={pdfLaTeX},%
  pdfproducer={LaTeX with hyperref}%
}
\newcommand{\floor}[1]{\left\lfloor{#1}\right\rfloor}

\newcommand{\abs}[1]{\left\lvert{#1}\right\rvert}

\newcommand{\mc}{\mathcal}
\newcommand{\ft}{\widetilde{f}}
\newcommand{\fh}{\widehat{f}}
\newcommand{\Th}{\widehat{T}_1}
\newcommand{\Tn}{T_{\boldsymbol{n}}}
\newcommand{\Uf}{\mathcal{U}_f}
\newcommand{\Up}{\mathcal{U}^{+}}
\newcommand{\Um}{\mathcal{U}^{-}}
\newcommand{\sU}{\mathcal{U}}
\newcommand{\sL}{\mathcal{L}}
\newcommand{\sC}{\mathcal{C}}
\newcommand{\sA}{\mathcal{A}}
\renewcommand{\rho}{\varrho}
\renewcommand{\phi}{\varphi}
\renewcommand{\epsilon}{\varepsilon}
\newcommand{\R}{\mathbb{R}}\newcommand{\N}{\mathbb{N}}
\newcommand{\Z}{\mathbb{Z}}\newcommand{\Q}{\mathbb{Q}}
\newcommand{\T}{\mathbb{T}}
\newcommand{\A}{\mathbb{A}}

 \newcommand{\ie}{i.e.\ }
\newcommand{\eg}{e.g.\ } 
\newtheorem{maintheorem}{Theorem}

\newtheorem{theorem}{Theorem}[section]
\newtheorem{proposition}[theorem]{Proposition}
\newtheorem{lemma}[theorem]{Lemma}
\newtheorem{corollary}[theorem]{Corollary}
\newtheorem*{coro}{Corollary}
\newtheorem{claim}[theorem]{Claim}

\theoremstyle{definition}

\theoremstyle{remark} \newtheorem{remark}[theorem]{Remark}

\newcommand{\Homeo}[1]{\mathrm{Homeo}_{#1}}

\newcommand{\dd}{\:\mathrm{d}}

\DeclareMathOperator{\M}{\mathfrak{M}} 
 
\DeclareMathOperator{\SL}{SL}
\DeclareMathOperator{\GL}{GL}

\newcommand*{\cl}[1]{\overline{#1}}

\DeclareMathOperator{\diam}{diam}
\DeclareMathOperator{\cc}{cc}

\DeclareMathOperator{\proj}{pr}
\newcommand{\pr}[1]{\proj_{#1}} 

\newcommand{\I}{\mathrm{I}}

\usepackage{tikz-cd}

\title[Bounded vertical deviations and irrational circle factors]{Bounded vertical deviations and irrational circle factors in Dehn Twist classes}

\author[H. Corrêa]{Heric Corrêa}
\address{IMPA.\@ Estrada Dona Castorina, 110. Rio de Janeiro, 22460--320, Brazil.}
\email{correaheric2@gmail.com}
\urladdr{\href{https://heric-correa.github.io/}{heric-correa.github.io}}
\thanks{H.C. was supported by CNPq-Brazil.}

\author[A. Kocsard]{Alejandro Kocsard}
\address{IMPA.\@ Estrada Dona Castorina, 110. Rio de Janeiro, 22460--320, Brazil.}
\email{alejandro@impa.br}
\urladdr{\href{https://w3.impa.br/~alejandro/}{w3.impa.br/~alejandro}}
\thanks{A.K. was supported by CNPq-Brazil.}


\date{\today}

\begin{document}

\begin{abstract}
  We prove that a homeomorphism of the two-torus in a Dehn twist homotopy class is a topological extension of an irrational circle rotation if and only if its vertical rotation set reduces to a single irrational number and the homeomorphism exhibits uniformly bounded rotational deviations in the vertical direction. This result has been previously obtained by the second author in~\cite{Kocsard2021} under an additional hypothesis about the geometry of the non-wandering set. We show that this hypothesis can be entirely dropped: the large-scale shear intrinsic to a Dehn twist eliminates the obstructions coming from the wandering set, which cannot be removed in the homotopy class of the identity. Finally, we prove that the factor is unique up to a circle rotation and is locally constant on the wandering set. Consequently, every connected component of the wandering set is contained in a single fiber, is itself a wandering domain and is either inessential or annular. 
\end{abstract}

\maketitle

\section{Introduction and statement of the results}

Given an orientation-preserving homeomorphism $g$ of the circle $\T=\R/\Z$, Poincaré’s rotation number assigns a single asymptotic speed to every orbit and when this number is irrational, the homeomorphism $g$ is a topological extension of the corresponding irrational rotation. In this setting, the structure of the non-wandering set is completely understood: either $g$ is minimal and then, conjugate to the corresponding rotation, or the non-wandering set $\Omega(g)$ is a Cantor set and is the only minimal set for $g$. In the latter case, the wandering set $\mc W(g)$ is a countable disjoint union of open wandering intervals. The purpose of this article is to establish a two-dimensional version of these pictures for torus homeomorphisms in a Dehn twist homotopy class.

For a homeomorphism of $\T^2$ homotopic to the identity, a lift may exhibit several asymptotic displacement vectors in the universal cover. Misiurewicz and Ziemian~\cite{MZ89} organized these vectors into the \emph{rotation set}, a nonempty compact convex subset of $\R^2$ whose points can also be described through invariant measures. Several dynamical properties of this kind of systems are reflected in the topological, geometric and arithmetical properties of the corresponding rotation set.

When the homeomorphism $f\colon\T^2\carr$ is homotopic to a Dehn twist, it is necessary to modify this dynamical invariant. In this setting only the vertical displacement satisfies the appropriate cocycle relation and so, as Doeff~\cite{Doeff1997} and Addas-Zanata~\cite{AddasZanataPropVertRotInt,AddasZanataPropVertRotIntTwistMapII} first observed, one just can define the \emph{vertical rotation set} in this setting. Among other results, they proved that if $\tilde f\colon\R^2\carr$ is a lift of $f$, its \emph{vertical rotation interval} $\varrho_V(\tilde f)$ (see Section~\ref{sec:vertical_rotation_interval} for a precise definition) is a compact interval whose rational elements in the interior of the set are realized by periodic orbits.

The main purpose of this paper is to study the dynamics of periodic point free homeomorphisms of $\T^2$ homotopic to a Dehn twist. If $f\colon\T^2\carr$ is such a homeomorphism and $\tilde f\colon\R^2\carr$ is a lift to the universal cover, by the results of Doeff and Addas-Zanata we mentioned above, the vertical rotation interval $\varrho_V(\tilde f)$ reduces to a single element $\varrho\in\R$. In this situation, one says that $f$ exhibits  \emph{$\varrho$-bounded vertical deviations} when there is $C>0$ satisfying
\begin{equation}
  \label{eq:bvd}
  \bigl|\pr{2}\bigl(\tilde f^{\,n}(\tilde z)-\tilde z\bigr)-n\varrho\bigr|\le C, \qquad\text{for every } \tilde z\in\R^2 \text{ and } n\in\Z .
\end{equation}

A priori, the condition of $\varrho$-bounded vertical deviations is stronger than the mere reduction of the vertical rotation interval to a single element. When the unique rotation number $\rho$ of $\tilde f$ is a  rational number, Addas-Zanata, Garcia and Tal~\cite{AddasZanataTalGarcia2014} have shown that $f$ indeed exhibits $\varrho$-bounded vertical deviations. In a paper that is still in preparation, we shall prove that the very same happens when $\rho$ is an irrational number.

On the other hand, given an arbitrary homeomorphism $f\in\Homeo{}(\T^2)$, an \emph{irrational circle factor} of $f$ is a pair $(h,\varrho)$ where $h\colon\T^2\to\T$ is a continuous surjection, $\varrho\in\R\setminus\Q$ and they satisfy $h\circ f=R_\varrho\circ h$. An irrational circle factor extracts a rigid component of the dynamics closely related to the \emph{pseudo-foliations} that the second author and Pereira-Rodrigues constructed in~\cite{KPR18} for homeomorphisms isotopic to the identity. These factors played a key role in recent progress on the Franks-Misiurewicz conjecture~\cite{PasseggiSambarino2020}.

Since every circle homeomorphism exhibits bounded rotational deviations (see \S\ref{sec:pre-notations} for details), one can easily show that any orientation-preserving homeomorphism of $\T^2$ that admits an irrational circle factor is either homotopic to the identity or it is conjugate to a homeomorphism homotopic to a Dehn twist. In the first case, its Misiurewicz-Ziemian rotation set must be contained in a line with rational slope, not containing any rational point and the homeomorphism must exhibit bounded rotational deviations in the direction perpendicular to that line (see for instance~\cite[Lemma~3.1]{JagerTal2017}). In the second case, assuming that the homeomorphism is already homotopic to a Dehn twist, the vertical rotation interval reduces to a single irrational number and the homeomorphism exhibits $\varrho$-bounded vertical deviations for some irrational $\varrho$ (see for instance~\cite[Proof of Corollary~1.2]{Kocsard2021}).

The converse of this relation is considerably much subtler, and previous results required some form of recurrence. Jäger showed that conservative totally irrational pseudo-rotations with bounded rotational deviations in every direction are semi-conjugate to the corresponding minimal translation~\cite[Theorem~C]{Jager2009}. Jäger and Tal characterized the existence of irrational circle factors for conservative torus homeomorphisms that are not eventually annular~\cite[Theorem~1.1]{JagerTal2017}; they also showed that bounded deviations do not suffice for homeomorphisms of the closed annulus, so annularity is a genuine obstruction. The second author~\cite[Theorem~A]{Kocsard2021} removed the conservativity assumption, replacing it with the hypothesis of \emph{small wandering domains}: every connected component of the wandering set $\mc W(f)$ is homotopically trivial, and only finitely many of them have diameter larger than any given positive number. In particular, for a homeomorphism homotopic to a nontrivial Dehn twist and having small wandering domains, the existence of an irrational circle factor is equivalent to $\varrho$-bounded vertical deviations with $\varrho$ irrational~\cite[Corollary~1.2]{Kocsard2021}.

In the homotopy class of the identity, this hypothesis on the wandering set cannot be removed. In~\cite{Kocsard2021}, the second author constructed three totally irrational pseudo-rotations with uniformly bounded rotational deviations in every direction which do not admit any irrational circle factor. Their wandering sets are, respectively, connected, inessential and lift-unbounded~\cite{Kocsard2021}; connected and fully essential; and a union of countably many lift-bounded discs, all of the same diameter (see \S\ref{sec:pre-notations} for precise definitions). Each example is obtained from a homeomorphism that does have a factor, by composing it with a homeomorphism supported in a wandering disc which moves a point to another point lying in a different fiber.

\subsection{Statements of the main results}

The examples above show that, in the homotopy class of the identity, bounded deviations with respect to a rational direction alone do not guarantee the existence of an irrational circle factor, some control on the wandering set and non-annularity are indispensable.

Our first result shows that these potential obstructions are automatically not present in Dehn twist homotopy classes:

\begin{maintheorem}
  \label{thm:main}
  Let $f\colon\T^2\carr$ be a homeomorphism homotopic to a Dehn twist, and let $\tilde f\colon\R^2\carr$ be a lift. If $\tilde f$ has $\varrho$-bounded vertical deviations for some $\varrho\in\R\setminus\Q$, then $f$ admits an irrational circle factor with rotation number $\varrho$.
\end{maintheorem}

So, without any assumption on the wandering set, we get the following complete characterization of the existence of irrational circle factors for homeomorphisms in Dehn twist homotopy classes:

\begin{coro}
  A homeomorphism $f\colon\T^2\carr$ homotopic to a Dehn twist admits an irrational circle factor if and only if any lift $\tilde f\colon\R^2\carr$ of $f$ has $\varrho$-bounded vertical deviations, for some $\varrho\in\R\setminus\Q$.
\end{coro}

The absence of hypotheses about the wandering set $\mc W(f)$ is not due to the inexistence of homeomorphisms with wandering sets containing large connected components. In fact, in Dehn twist homotopy classes there exist homeomorphisms admitting irrational circle factors and such that its wandering sets contain essential connected components. For instance, if $g$ is a Denjoy counterexample on $\T$ and $k\neq0$, the skew product $f(x,y)=(x+ky,\,g(y))$ is homotopic to the Dehn twist of coefficient $k$, has bounded vertical deviations with irrational $\varrho$, and one checks that $\mc W(f)=\T\times\mc W(g)$ is a disjoint union of essential open annuli. Such a map does not have small wandering domains in the sense of~\cite{Kocsard2021}.

Our next result shows that the semi-conjugacy map $h$ given by Theorem~\ref{thm:main} is essentially unique:

\begin{maintheorem}
  \label{thm:main-semi-conjugacy-uniqueness}
  If $f$ and $\tilde f$ are as in Theorem~\ref{thm:main} and suppose that $\tilde f$ exhibits $\varrho$-bounded vertical deviations for some $\varrho\in\R\setminus\Q$. Then, if $(h_1,\varrho)$ and $(h_2,\varrho)$ are two irrational circle factors of $f$, then there exists $\alpha\in\T$ such that $h_1(z)=h_2(z)+\alpha$ for every $z\in\T^2$.
\end{maintheorem}

Returning to the analogy with the one-dimensional rotation theory, it is natural to ask whether the existence of irrational circle factors in Theorem~\ref{thm:main} also imposes restrictions on the topology and the dynamics of the wandering set $\mc W(f)$. Our third main result shows that this is indeed the case:

\begin{maintheorem}
  \label{thm:main-wandering-set}
  Let $f$, $\widetilde f$ and $\varrho$ be as in Theorem~\ref{thm:main}, and let $(h,\varrho)$ be an irrational circle factor for $f$. Then $h$ is locally constant on $\mathscr{W}(f)$ and
  \begin{displaymath}
    \mathscr{W}(f)=\bigsqcup_{\theta\in\mathbb{T}}\operatorname{int}\bigl(h^{-1}(\theta)\bigr).
  \end{displaymath}
  In particular, every connected component of the wandering set $\mathscr{W}(f)$ is a wandering set, is contained in a single fiber of $h$, and is either inessential or annular. So, the non-wandering set $\Omega(f)$ is essential in $\T^2$.
\end{maintheorem}

Thus, exactly as for Denjoy counterexamples on the circle, the wandering set coincides with the union of the interiors of the fibers of the semi-conjugacy, and every wandering domain is collapsed by the semi-conjugacy to a single point of $\T$.

The conclusion of Theorem~\ref{thm:main-wandering-set} fails in the identity homotopy class, even under the existence of an irrational circle factor. Consider the map $g$ used as the starting point of~\cite[\S3.3]{Kocsard2021}: it is the ``time-$\tilde g_1$'' map of the suspension flow of a Denjoy homeomorphism $g_2$, where $g_1$ is another Denjoy homeomorphism. By~\cite[Proposition~3.1]{Kocsard2021}, for a suitable choice of rotation numbers, $g$ is a totally irrational pseudo-rotation with uniformly bounded rotational deviations in every direction (so it satisfies~\eqref{eq:bvd} with irrational $\varrho$), it is a topological extension of a minimal translation of $\T^2$, and its wandering set is connected and fully essential. From the description of $\mc W(g)$ in~\cite[Proposition~3.1(v)]{Kocsard2021} one readily checks that $\Omega(g)$ is a compact totally disconnected set: it is contained in a Cantor family of circles, and meets each of them in a Cantor set. Hence $\Omega(g)$ is inessential. Moreover, since $\mc W(g)$ is connected and invariant, its image under any irrational circle factor is a connected $R_\varrho$-invariant subset of $\T$, hence equal to $\T$; so no irrational circle factor of $g$ is locally constant on $\mc W(g)$.

\subsection{Strategy of the proofs} The proof of Theorem\ref{thm:main} starts with Jäger's ordered equivariant family $\{\mc C_r\}_{r\in\R}$ of circloids on the annulus~\cite{Jager2009}. The main step is to show that different elements of this family are indeed disjoint. If two of them intersect, they bound a \emph{bunch} to which the planar copy-counting results of Passeggi and Sambarino apply~\cite[Propositions~2.2 and~2.5]{PasseggiSambarino2020}. Then, choosing an irrational near return $q\varrho-m=\delta\neq0$, for suitable lifts $\tilde z,\tilde w$, the Dehn twist turns the vertical error $\delta$ into horizontal growth:
\[
  \bigl|\pr{1}\bigl(\tilde f^{\,n}(\tilde w)-\tilde f^{\,n}(\tilde z)\bigr)-nk\delta\bigr|\le K \qquad\text{for every } n\ge0,
\]
where $k\neq0$ is the twist coefficient. Normalized returns keep the growing continua inside a fixed bunch. Copy-counting then forces arbitrarily many horizontal deck translates, whereas pulling them back to the fixed initial continuum gives a uniform bound. This contradiction proves disjointness, and the order of the circloids defines the factor map onto $R_\varrho$. The shear is exactly what replaces the two distinct rotation vectors used by Passeggi and Sambarino in the identity class.

The proof of Theorem~\ref{thm:main-wandering-set} is based on a perturbative argument which turns second author's construction of counterexamples in the identity class into a rigidity mechanism. First, since $f$ is periodic point free and not eventually annular (Lemma~\ref{lemma:our_f_everypoint_is_fullyessential}), every non wandering point is fully essential in the sense of Koropecki and Tal~\cite{KoropeckiTal2014,Kocsard2021}, and this implies that $\Omega(f)$ is weakly transitive (Lemma~\ref{lemma:our_f_is_weaklytransitive}). Assuming now that $h$ is not locally constant on some wandering disc $U$, we compose $f$ with a diffeomorphism $\ell$ supported on $U$, which pushes a segment along which $h$ increases (Lemma~\ref{lemma:the_local_pertubation}). The new map $g=f\circ\ell$ still has $\varrho$-bounded vertical deviations (Lemma~\ref{deviation_estmate_of_ppertubed}), so Theorem~\ref{thm:main} applies to $g$ as well. This is the step where the Dehn twist hypothesis is used once more, and it is precisely what fails on the identity homotopy class. The difference between the lifts of the factors of $f$ and $g$ descends to a continuous function $u$ with $u\circ g-u=\delta$, where $\delta$ vanishes outside the wandering disc $U$. A Gottschalk--Hedlund type obstruction (Lemma~\ref{lemma:coboundary_obstruction}) forces $\max\delta=-\min\delta$, whereas the choice of $\ell$ gives $\max\delta>-\min\delta$. This contradiction shows that $h$ is locally constant on $\mc W(f)$ (Proposition~\ref{prop:h_is_constant}), and the description of $\mc W(f)$ follows.

\subsection{Organization of the paper} This paper is organized as follows. Section~\ref{sec:pre-notations} introduces the notation and preliminary results concerning vertical rotation sets, essentiality and annularity. It also recalls Jäger's ordered family of circloids and explains how its pairwise disjointness yields an irrational circle factor. Section~\ref{sec:proof_main_thm} proves Theorem~\ref{thm:main} by combining the planar continuum estimates of Passeggi-Sambarino with the shear induced by the Dehn twist. Section~\ref{sec:proof-main-semi-conjugacy-uniqueness} establishes the uniqueness of the factor up to a circle rotation, proving Theorem~\ref{thm:main-semi-conjugacy-uniqueness}. Finally, Section~\ref{sec:proof-wandering-set} develops the local perturbation argument and proves Theorem~\ref{thm:main-wandering-set}, describing the wandering set in terms of the interiors of the fibers. 

\section*{Acknowledgments}
The authors would like to thank Alejandro Passeggi for bringing his article~\cite{PasseggiSambarino2020} to their attention and suggesting its application to the proof of Theorem~\ref{thm:main}. This is part of the first author's PhD thesis, which is being carried out at IMPA under the supervision of the second author and is supported by a CNPq-Brazil PhD scholarship. The second author is partially supported by CNPq-Brazil grant \emph{``Bolsa de Produtividade em Pesquisa''.}

\section{Preliminaries and notations}
\label{sec:pre-notations}

\subsection{General topology}

Let $X$ be a complete metric space. Given any $A\subset X$, we write $\mathrm{int}(A)$ for its interior in $X$ and $\overline{A}$ for its closure. If $A$ is connected, and $A \subset Y \subset X$, we write $\cc(Y,A)$ for the connected component of $Y$ containing $A$. The group of self-homeomorphisms of $X$ is denoted by $\Homeo{}(X)$. Given any $f \in \Homeo{}(X)$, its \emph{support} is the closed set
\[\mathrm{supp} \ f \coloneq \overline{\left\{ x \in X : f(x) \neq x \right\}}.\]
If $X$ is a compact metric space, we denote by $C^0(X;\R)$ the space of continuous functions $\psi \colon X \to \R$ endowed with the usual $\sup$ norm $\|\psi\|_{\infty} := \sup\{|\psi(x)| : x \in X\}$.

For any Cartesian product $X \times Y$ between metric spaces, we denote by $\pr{1}$ and $\pr{2}$ the projections onto the first and second coordinates, respectively.

\subsection{Dynamics} Given $f \in \Homeo{}(X)$, the \emph{orbit} of $x \in X$ under $f$ is the set $ \mathcal{O}_f(x) \coloneq \{f^n(x) : n \in \Z\}$. A point $x \in X$ is \emph{fixed} under $f$ if $\mathcal{O}_f(x) = \{x\}$, and \emph{periodic} if $\mathcal{O}_f(x)$ is finite; we say that $f$ is \emph{periodic-point-free} if it has no periodic points. Furthermore, $f$ is said to be \emph{minimal} if the orbit of every point is dense in $X$.

\subsubsection{Recurrent and non-wandering points}

Given $f \in \Homeo{}(X)$ and $x \in X$, let $\omega_f(x)$ (respectively, $\alpha_f(x)$) denote the set of points $y\in X$ for which there exists a sequence $n_j\to+\infty$ (respectively, $n_j\to-\infty$) with $f^{n_j}(x)\to y$.

An open set $U \subset X$ is said to be \emph{wandering} for $f$ when $f^n(U) \cap U = \emptyset$, for every $n \in \Z \setminus\{0\}$. A point $x \in X$ is called \emph{non-wandering} for $f$ when no neighborhood of $x$ is wandering for $f$. The \emph{non-wandering set} of $f$ is the set of non-wandering points and is denoted by $\Omega(f)$. It is closed, $f$-invariant, and contains $\omega_f(x)$ and $\alpha_f(x)$ for every $x\in X$. Its complement, called the \emph{wandering set} of $f$, is denoted by $\mathcal{W}(f)$. When $X$ is locally connected, a connected component of $\mc W(f)$ is called a \emph{wandering domain} when it is a wandering set for $f$. An open topological disc which is a wandering set is called a \emph{wandering disc.}

\subsubsection{Transitivity and minimality} Given $f \in \Homeo{}(X)$, a set $K\subset X$ is said to be \emph{$f$-invariant} if $f(K)=K$. A nonempty compact $f$-invariant set $K$ is a \emph{minimal set} for $f$ if it contains no proper nonempty compact $f$-invariant subset.

A compact $f$-invariant set $K\subset X$ is said to be \emph{weakly transitive} if for every pair of open sets $A,B\subset X$ meeting $K$, i.e.\ $A\cap K\neq\emptyset$ and $B\cap K\neq\emptyset$, there exists $n\geq1$ such that $f^n(A)\cap B\neq\emptyset$. When $f^n(A)\cap B\neq\emptyset$ for some $n\geq 1$ and every pair of nonempty relatively open sets $A,B\subset K$, we say that $K$ is \emph{topologically transitive}.

We shall need the following lemma in the proof of Theorem~\ref{thm:main-wandering-set}:

\begin{lemma}
\label{lemma:coboundary_obstruction}
Let $f\colon X \carr$ be a homeomorphism of a compact metric space $X$ and suppose that the non-wandering set  $\Omega(f)$ is weakly transitive. Let $U \subset X$ be an open wandering set for $f$. Let us assume there exists $u \in C^0(X,\R)$ such that the function $\delta \coloneq u\circ f - u$ vanishes outside $U$.
Then
\[\max_{x \in X} \delta(x) = - \min_{x \in X}\delta(x).\]
\end{lemma}

\begin{proof}
Let us write $M \coloneq \max \delta$ and $m \coloneq \min \delta$. We start showing that $M \leq -m$. First, note that $m \leq 0$ because $\delta(x) = 0$ for every $x \in X \setminus U$. For each $n \geq 1$ and every $x \in X$, we have
\begin{equation}
\label{eq:iteratin_coboundary}
    u(f^n(x)) - u(x) = \sum_{j=0}^{n-1} \delta(f^j(x)) \geq m.
\end{equation}
Since the sets $f^n(U)$ are pairwise disjoint, every orbit intersects $U$ at most once. Consequently, the sum on the right hand side of~\eqref{eq:iteratin_coboundary} contains at most one nonzero term.

Reasoning by contradiction, assume that there exists $x_0 \in X$ such that $a \coloneq \delta(x_0) > -m$. Then $a > 0$ and $x_0 \in U$. Thus,
\begin{equation}\label{eq:u_constant}
   u(f^n(x_0)) = \begin{cases} u(x_0),  &\quad n \leq 0, \\
    u(x_0) + a, &\quad n \geq 1.
    \end{cases}
\end{equation}

By compactness, both limit sets $\alpha_f(x_0)$ and $\omega_f(x_0)$ are nonempty. Choosing points $x_1 \in \alpha_f(x_0)$ and $x_2 \in \omega_f(x_0)$, we know that $x_1,x_2 \in \Omega(f)$. Taking limits in \eqref{eq:u_constant} yields $u(x_1) = u(x_0)$ and $u(x_2) = u(x_0)+a$.  Since $-a < m$, we can find, by continuity, neighborhoods $V_1$ of $x_1$ and $V_2$ of $x_2$ such that
\[u(x) - u(y) < m \quad \text{for every} \ (x,y) \in V_1 \times V_2.\]

Now, since $V_1 \cap \Omega(f) \neq \emptyset$ and $V_2 \cap \Omega(f) \neq \emptyset$, weakly transitivity of $\Omega(f)$ gives $n \geq 1$ and $z \in V_2$ such that $f^n(z) \in V_1$. Then $u(f^n(z))-u(z) < m$ contradicting \eqref{eq:iteratin_coboundary}.

For the other inequality, note that $u' \coloneq -u$ satisfies the same hypotheses, with $\delta' = -\delta$. Repeating the above reasoning to $u'$, one gets $-m \leq M$.
\end{proof}

\subsubsection{Orbit of open sets}
Let $X$ be a locally connected complete metric space and let $f \in \Homeo{}(X)$. Given any nonempty connected open set $V \subset X$, following \cite{KoropeckiTal2014} and
\cite{Kocsard2021}, we define
\begin{equation}
\label{eq:orbit-open-set}
\mathcal{U}_f(V) \coloneq
\cc\left(\bigcup_{n \in \Z} f^n(V), V\right).
\end{equation}

We say that a nonempty open set $V \subset X$ is \emph{periodic for $f$} if $f^q(V)=V$ for some integer $q \geq 1$. A function $\phi \in C^0(X;\R)$ is called \emph{$f$-invariant} if $\phi \circ f=\phi$. Given $c\in\R$ and $\star\in\{<,\leq,=,\geq,>\}$, we write
$\{\phi\star c\}\coloneq\{x\in X:\phi(x)\star c\}$.

\begin{lemma} \label{lemma:U-c-cc-periodic}
\label{lem:periodic-level-components}
Let $X$ be a compact, locally connected metric space, let $f\colon X \carr$ be a homeomorphism, and let $\phi \in C^0(X;\R)$ be $f$-invariant. Then, for every $c \in \R$, each connected component of $\{\phi<c\}$ and of $\{\phi>c\}$ is periodic for $f$.
\end{lemma}
\begin{proof}
Let $V$ be a connected component of $\{\phi<c\}$, choose $x\in V$, and set $a\coloneq\phi(x)<c$. The set $K \coloneq \{\phi \leq a\}$ is compact and contained in $\{\phi<c\}$. Moreover, since $\phi$ is $f$-invariant, $f^n(x)\in K$ for every $n\in\Z$. By local connectedness, each point of $K$ has a connected open neighborhood contained in $\{\phi<c\}$. By compactness, finitely many such neighborhoods cover $K$. Each of them is contained in a single connected component of $\{\phi<c\}$, so $K$ meets only finitely many components of $\{\phi<c\}$. Since $f$ is a homeomorphism and $\{\phi<c\}$ is $f$-invariant, $f$ permutes its connected components. For every $n\geq0$, the component $f^n(V)$ meets $K$ at $f^n(x)$. Thus the family $\{f^n(V):n\geq0\}$ is finite, and there exist integers $0\leq r<s$ such that $f^r(V)=f^s(V)$. Applying $f^{-r}$ yields
\[ f^{s-r}(V)=V,\]
which proves that $V$ is periodic.

Applying the same argument to $-\phi$ proves the assertion for the connected components of $\{\phi>c\}$.
\end{proof}

\subsection{Homeomorphisms of the torus}  Let us write $\T^{d} := \R^{d}/\Z^{d}$ for the $d$-dimensional torus. As usual, we simply write $\T$ instead of $\T^1$. Let $\A := \T \times \R$ be the open $2$-dimensional annulus. Let $\pi_{\T^d}\colon\R^{d}\to\T^{d}$ and $\pi_{\A}\colon \R^2 \to \A$ denote their respective natural quotient projections. Whenever there is no risk of ambiguity, we will simply denote both $\pi_{\T^d}$ and $\pi_{\A}$ by $\pi$. We also define the covering map $\hat{\pi}\colon \A \to \T^2$ by $\hat{\pi} = \mathrm{id}\times \pi_{\T}$.

As a notational convention, points in $\R$ will be denoted using the ``tilde notation'': $\tilde x \in \R$, and a we will simply write $x = \tilde x + \Z$ for the corresponding point in $\T$. Moreover, given points $(\tilde x, \tilde y) \in \R^2$, $(x,\tilde y) \in \T \times \R$ and integers $m, \ell \in \Z$, we write $T_{(m,\ell)}$ and $T_1$ for the translations given by
\[T_{(m,\ell)}(\tilde x, \tilde y) := (\tilde x + m, \tilde y+\ell), \qquad \widehat T_1(x,\tilde y) := (x,\tilde y + 1).\]
Also, as it is usually done, any continuous $\Z^d$-periodic function $\psi\colon\R^d\to\R^n$ shall be identified with the corresponding continuous function $\psi\colon\T^d\to\R^n$.

We shall write $\Homeo{+}(\T^d)$ for the group of orientation-preserving homeomorphisms of $\T^d$. It is well known that every $A\in\GL_{d}(\Z)$ is the lift of a Lie group automorphism of $\T^d$, which, by an abuse of notation, we also denote by $A\colon \T^{d}\carr$. Conversely, for any homeomorphism $f\in\Homeo{}(\T^{d})$, there exists a unique $A_f\in\GL_{d}(\Z)$ such that $f$ is homotopic to the corresponding Lie group automorphism $A_f$. In this case, one can easily show that for any lift $\tilde f\colon\R^{d}\carr$ of $f$, the difference $\tilde f-A_f\colon\R^{d}\to\R^{d}$ is $\Z^{d}$-periodic. Consequently, it can be viewed as an element of $C^{0}(\T^{d},\R^{d})$.

\subsubsection{Periodic-point free homeomorphisms} As we aim to find irrational rotation factors for $f$, we can restrict ourselves to the set of periodic-point-free homeomorphisms of $\T^2$. For that, as a straightforward consequence of the Lefschetz fixed-point theorem, given any $A \in \mathrm{SL}_2(\Z)$, the homotopy class $\Homeo{A}(\T^2)$ contains periodic-point-free homeomorphisms if and only if $A$ is unipotent, \ie~$1$ is the only eigenvalue of $A$. By classical Jordan decomposition theorem (see \eg \cite[Proposition 2.4]{Kocsard2021} for the calculations), it follows that for any periodic-point-free homeomorphism $f\in\Homeo{+}(\T^{2})$, there exist $B\in\SL_{2}(\Z)$ and a unique $k\in\Z$ such that
\begin{displaymath}
     BA_{f}B^{-1}=I_{k}:=
  \begin{pmatrix}
    1 & k\\
    0 & 1
  \end{pmatrix}.
\end{displaymath}
Consequently, $B\circ f\circ B^{-1}\in \Homeo{k}(\T^{2})$, where, for simplicity, we write $\Homeo{k}(\T^{2})$ to denote $\Homeo{I_{k}}(\T^{2})$. Therefore, when studying the dynamics of periodic-point-free homeomorphisms on $\T^{2}$, there is no loss of generality in restricting our attention to homeomorphisms in $\Homeo{k}(\T^{2})$ for some $k\in\Z$. The toral automorphism associated to \(\I_{k}\), with $k \neq 0$, is called a \emph{Dehn twist} and its homotopy class is named accordingly.

We also denote by $\Homeo{0}(\A)$ the group of homeomorphism of the annulus homotopic to the identity. Given $f \in \Homeo{k}(\T^2)$, any lift $\tilde f\colon \R^2 \carr$ induces a unique $\hat f \in \Homeo{0}(\A)$ which is a lift of $f$ with respect to $\hat \pi$. Finally, since $\tilde f \circ T_{(0,1)} = T_{(k,1)}\circ \tilde f$, we note that  $\hat f \circ \hat{T}_1 = \hat T_1 \circ \hat f$.

\subsubsection{Vertical rotation interval}\label{sec:vertical_rotation_interval} Given $f \in \Homeo{k}(\T^2), k\neq 0$, and a lift \(\tilde{f} \colon \R^2 \carr\) of \(f\), its \emph{displacement map} is $\Delta_{\tilde f} := \tilde{f} - \I_k$. More generally, for any $n \in \Z$, the iterate $f^n$ belongs to $\Homeo{kn}(\T^2)$ and we denote its displacement map by \[\Delta_{\tilde f}^{(n)} := \Delta_{\tilde f^n} = \tilde f^n - I_{kn}.\]
We write the coordinate functions of $\Delta_{\tilde f}^{(n)}$ by $\Delta_i^{(n)}\colon \R^2 \to \R$ where \[\Delta_i^{(n)} = \pr{i}\circ \Delta^{(n)} \quad \text{for each}\ i \in \{1,2\}.\]

For the coordinate functions, we are omitting the dependence of the lift $\tilde f$ just for the simplicity of the notation. Moreover, since $\Delta_{\tilde f}^{(n)}$ is $\Z^2$-periodic, we can consider $\Delta_{i}^{(n)} \in C^0(\T^2; \R)$ for any $i \in \{1,2\}$.

A simple induction argument shows that
\begin{equation}\label{eq:cocycle_properties}
\Delta_2^{(n)} = \sum_{j=0}^{n-1} \Delta_2\circ f^j.
\end{equation}
This cocycle relation does not hold for the first coordinate function $\Delta_1^{(n)}$, which is the primary reason we cannot define a two-dimensional rotation set analogous to the case of homeomorphisms homotopic to the identity as done in the seminal work \cite{MZ89} of Misiurewicz and Ziemian.

By Birkhoff's ergodic theorem, it is thus natural to define the following set, called the \emph{vertical rotation set} of $\tilde f$:
\begin{equation}
  \label{eq:vertical-rotation-set-def}
  \varrho_V(\tilde{f}) := \left\{ \int_{\T^{2}}\Delta_2(z) \, \dd\mu (z) \, \colon \, \mu \in \M(f) \right\} \text{,}
\end{equation}
where $\M(f)$ stands for the set of $f$-invariant probability measures endowed with the weak-$\star$ topology. The set \eqref{eq:vertical-rotation-set-def} is the image of a compact, convex and nonempty set under a linear and continuous map, being therefore a nonempty and compact interval of the real line. A topological way to characterize it is as the following limit set:
\begin{equation}
  \label{eq:vertical-rotation-set-equiv}
  \varrho_V(\tilde{f}) = \bigcap_{m \geq 1} \cl{ \bigcup_{n \geq m} \left\{ \frac{\Delta_2^{(n)} (z) }{n} \, \colon \, z \in \T^2 \right\} } \text{.}
\end{equation}
This equality may be established in a way analogous to that of \cite[Corollary 3.5]{MZ89}.

\subsection{Inessential and essential sets}\label{sec:essential_sets}

Let $U\subset\mathbb{T}^2$ be a nonempty connected open set, and let $i\colon H_1(U,\mathbb{Z})\to H_1(\mathbb{T}^2,\mathbb{Z})$ denote the group homomorphism induced on first homology groups by the inclusion. The set $U$ is said to be \emph{inessential} when $i=0$ (equivalently, every closed curve in $U$ is homotopically trivial in $\mathbb{T}^2$, since $\pi_1(\mathbb{T}^2)$ is abelian), and \emph{essential} otherwise. The open connected set $U$ is called \emph{annular} when the image of $i$ has rank~$1$, and \emph{fully essential} when $i$ is surjective.

An arbitrary set $E\subset\mathbb{T}^2$ is said to be \emph{inessential} when there exists an inessential open set containing $E$; otherwise, it is said to be \emph{essential}. Analogously, $E$ is said to be \emph{annular} when it is essential and there exists an annular open set containing $E$, and $E$ is said to be \emph{fully essential} when $\mathbb{T}^2\setminus E$ is inessential. These notions follow Koropecki and Tal \cite[Section 1.4 and Proposition 1.2]{KoropeckiTal2014}, and for open sets they agree with the homological ones given above. In particular, a closed set $K\subset\mathbb{T}^2$ is essential if and only if $\mathbb{T}^2\setminus K$ is not fully essential, that is, if and only if $\mathbb{T}^2\setminus K$ is either
inessential or annular.

If $f \in \Homeo{+}(\T^2)$ and $z \in \T^2$, we say that $z$ is an \emph{inessential} point for $f$ when there exists a neighborhood $V$ of $z$ such that the connected open set $\mathcal{U}_f(V)$ given by~\eqref{eq:orbit-open-set} is inessential; otherwise, it is said to be \emph{essential point}. Moreover, $z$ is a \emph{fully essential} point for $f$ when $\Uf(V)$ is fully essential for every connected open neighborhood $V$ of $z$.

\subsection{Annularity of torus homeomorphisms} We say that $f \in \Homeo{+}(\T^2)$ is \emph{annular} when there exists a lift $\tilde f\colon \R^2 \carr$ of $f$, a unitary vector $v \in \R^2$ with rational slope and a constant $C > 0$ such that \[\abs{\langle \tilde f^n(\tilde z) - \tilde z, v \rangle} \leq C \quad \text{for all}\ \tilde z \in \R^2, n \in \Z.\]
Here, $\langle\cdot, \cdot \rangle$ stands for the usual inner product in $\R^2$. More generally, $f$ is said to be \emph{eventually annular} when there exists $q \geq 1$ such that $f^q$ is annular. We will need the following proposition from~\cite[Proposition 2.11]{Kocsard2021}.

\begin{proposition}\label{prop:kocsard_eventuallyannular}
    If $f \in \Homeo{+}(\T^2)$ is periodic-point free and is not eventually annular, then every non-wandering point $x \in \Omega(f)$ is fully essential for f.
\end{proposition}

The homeomorphisms considered in this paper satisfy both hypotheses of Proposition~\ref{prop:kocsard_eventuallyannular}. Heuristically, the shear of a nontrivial Dehn twist forces horizontal displacements to be unbounded, so the only possible direction of annularity for an iterate of $f$ is the vertical one; but in that direction orbits drift linearly with nonzero speed. This is the content of the following lemma.

\begin{lemma}\label{lemma:our_f_everypoint_is_fullyessential}
  If $f \in \Homeo{k}(\T^2)$, $k \neq 0$, has $\varrho$-bounded vertical deviation with $\varrho \in \R \setminus \Q$, then $f$ is periodic-point-free and is not eventually annular.
\end{lemma}
\begin{proof}
  It is straightforward to see that $\varrho$-bounded vertical deviation implies that $\varrho_V(\tilde f) = \{\varrho\}$. Consequently $f$ is periodic-point-free. Indeed, if $z$ were a periodic point of period $p\geq1$, then $\tilde f^{p}(\tilde z)=\tilde T_{(a,b)}(\tilde z)$ for some $\tilde z\in\pi^{-1}(z)$ and $(a,b)\in\Z^2$, and the invariant measure supported on the orbit of $z$ would have vertical rotation number $b/p\in\Q$. 

    Assume now, toward a contradiction, that $f^q$ is annular for some $q\in\mathbb{N}$: there exist a lift $F$ of $f^q$, a unit vector $v=(v_1,v_2)$ with rational slope and $C'>0$ such that \[|\langle F^n(\tilde z)-\tilde z,v\rangle|\leq C'\] for all $\tilde z\in\R^2$ and $n\in\Z$. Since $f^q\in\Homeo{qk}(\T^2)$, we have $F(T_{(0,m)}\tilde z)=T_{(qkm,m)}F(\tilde z)$, so the displacements of $F$ at $\tilde z$ and at $T_{(0,m)}(\tilde z)$ differ by $(qkm,0)$. Thus the corresponding inner products with $v$ differ by $qkmv_1$. Taking $n=1$ and letting $m$ range over $\Z$, the fact that $k\neq0$ forces $v_1=0$, that is, $v=(0,\pm1)$.

     On the other hand, every lift of $f^q$ has the form $F=T_{(a,b)}\circ\tilde f^{q}$ for some $(a,b)\in\Z^2$. Since $I_{qk}$ preserves second coordinates, we get
\[ \pr{2}\bigl(F^n(\tilde z)-\tilde z\bigr)=\pr{2}\bigl(\tilde f^{\,qn}(\tilde z)-\tilde z\bigr)+nb, \]
and \eqref{eq:bvd} yields $\bigl|\pr{2}(F^n(\tilde z)-\tilde z)-n(q\varrho+b)\bigr|\leq C$ for all $n\in\Z$. As $\varrho$ is irrational, $q\varrho+b\neq0$, so $\pr{2}(F^n(\tilde z)-\tilde z)$ grows linearly with $n$. This contradicts the annularity of $f^q$ with $v=(0,\pm1)$.
\end{proof}

\subsection{Annular continua and circloids}\label{sec:annular_continua_and_circloids}

An \emph{annular continuum} $A\subset\A$ is a compact connected set whose complement consists of exactly two connected components, both unbounded.  They are denoted by $\mathcal U^-(A)$ and $\mathcal U^+(A)$, according to whether they contain the lower or the upper end of $\A$.  A \emph{circloid} is an annular continuum minimal for inclusion among annular continua.

For an annular continuum $A$ and a set $B\subset\A$, we write $A\preceq B$ when $B\subset A\cup\mc U^+(A)$, and $A\prec B$ when $B\subset\mc U^+(A)$. If $A$ and $B$ are annular continua and $A\preceq B$, we consider
\[ [A,B]:=\A\setminus\bigl(\mc U^-(A)\cup\mc U^+(B)\bigr).
\]
The relation $\preceq$ is transitive, is preserved by end-preserving annulus homeomorphisms, and is antisymmetric on circloids. For circloids $C,C'$, one has $C\preceq C'$ if and only if $C\subset C'\cup\mathscr{U}^-(C')$. If $A\preceq B$, then $[A,B]$ is an annular continuum. A circloid $C$ satisfying $A\preceq C\preceq B$ is contained in $[A,B]$. If $C\preceq C'$ are circloids, then $\mc U^+(C')\subset\mc U^+(C)$ and $\mc U^-(C)\subset\mc U^-(C')$. Finally, two distinct comparable circloids $C\preceq C'$ are disjoint exactly when the order is strict, that is, $C\prec C'$. These standard facts follow from separation in the two-point compactification of the annulus; see \cite[Section 3]{Jager2009} and \cite[Section 2]{PasseggiSambarino2020}.

\subsection{The ordered circloid family}
\label{sec:ordered_circloid_family}

In this section, for the sake of completeness, we explain and summarize the well-known construction of the circloids due to Jäger \cite{Jager2009}. Henceforth, we assume that $f \in \Homeo{k}(\T^2)$, $k \neq 0$, has $\varrho$-bounded vertical deviation for some $\varrho \in \R\setminus \Q$.  We denote by $\hat f \in\Homeo{0}(\A)$ the induced annulus homeomorphism.

For each $r\in\R$, we consider the set $\mathcal{A}_r \subset \A$ given by
\begin{equation}\label{eq:Ar}
 \mathcal A_r =
 \bigcup_{n\in\Z}\widehat f^{\ n}
 \bigl( \T\times\{r-n\varrho\}\bigr).
\end{equation}
The deviation estimate \eqref{eq:bvd} gives $\mathcal A_r\subset \T\times[r-C,r+C]$ for every $r \in \R$. Also $\T \times \{r\} \subset\mathcal A_r$, so $\overline{\mathcal A_r}$ is essential and is a \emph{lower generating set} in the terminology of J\"ager~\cite[Section 3]{Jager2009}, \ie a set which is bounded from above and has essential closure.

For completeness, we recall the operations involved. If $L$ is a lower generating set, let $\sU(L)$ be the connected component of $\A\setminus\overline{L}$ which is unbounded from above; if $U$ is an upper generating set, let $\sL(U)$ be the connected component of $\A\setminus\overline{U}$ which is unbounded from below. These sets are again upper and lower generating sets, respectively, so the compositions $\sL\sU(L):=\sL(\sU(L))$ and $\sU\sL\sU(L):=\sU(\sL(\sU(L)))$ make sense. For an annular continuum $E$ one has $\sU(E)=\Up(E)$ and $\sL(E)=\Um(E)$. Jäger's \emph{upper circloid hull} of $L$ is
\begin{equation}\label{eq:hull}
\sC^+(L):=\A\setminus\bigl(\sL\sU(L)\cup\sU\sL\sU(L)\bigr).
\end{equation}
By \cite[Lemma 3.2]{Jager2009}, the set $\sC^+(L)$ is a circloid. Define
\begin{equation}\label{eq:Cr}
\sC_r:=\sC^+(\sA_r).
\end{equation}

The following proposition is the construction in Jäger \cite[equations (4.1)-(4.6)]{Jager2009}. Although \cite{Jager2009} deals with homeomorphisms homotopic to the identity, the proofs of these properties only use the deviation estimate \eqref{eq:bvd} and the facts that $\fh$ is end-preserving and commutes with $\Th$, so they apply in our setting.

\begin{proposition}\label{prop:jager-family}
There exists $\tau>0$ \textup{(}in fact, one may take $\tau=C$\textup{)} such that, for all $r,s\in\R$,
\begin{align}
 \mathcal C_r&\subset\T\times[r-\tau,r+\tau],
                                                        \label{eq:localization}\\
 \mathcal C_{r+1}&=\widehat T_1(\mathcal C_r),                    \label{eq:integer-equiv}\\
 \mathcal C_{r+\varrho} &=\widehat f(\mathcal C_r),       \label{eq:f-equiv}\\
 r<s&\Longrightarrow \mathcal C_r\preceq\mathcal C_s.\label{eq:weak-order}
\end{align}
\end{proposition}

The next elementary fact prevents equality from being confused with overlap.

\begin{lemma}\label{lem:distinct}
If $r\ne s$, then $\mathcal C_r\ne\mathcal C_s$.
\end{lemma}

\begin{proof}
Suppose $r<s$ and $\mathcal C_r=\mathcal C_s$.  By \eqref{eq:weak-order}, $\mathcal C_t = \mathcal{C}_r$ for every $t\in[r,s]$.  Put $d=s-r>0$ and $G=\Z+\varrho\Z.$ The irrationality of $\varrho$ makes $G$ dense.  Equations \eqref{eq:integer-equiv} and \eqref{eq:f-equiv}, together with commutation of $\widehat T_1$ and $\widehat f$, imply
\begin{equation}\label{eq:G-action}
 \mathcal C_{u+p+q\varrho}=\widehat T^p\widehat f^{\ q}(\mathcal C_u)
 \quad \text{for every}\ u \in \R, \ p,q\in\Z.
\end{equation}
Consequently the family is constant on every translate $[r+g,s+g]$, $g\in G$.

Choose $\varepsilon\in G$ with $0<\varepsilon<d$.  Consecutive intervals
$[r+j\varepsilon,s+j\varepsilon]$, $j\in\Z$, overlap and their union is $\R$.  Their constant values agree on the overlaps.  Hence $\mathcal C_t$ would be independent of $t$ on all of $\R$.  This contradicts the localization \eqref{eq:localization}, since a fixed nonempty compact set cannot lie in $\T\times[t-\tau,t+\tau]$ for every $t$.
\end{proof}

Note that Lemma \ref{lem:distinct} and condition \eqref{eq:weak-order} together do not imply that $\mathcal{C}_r \cap \mathcal{C}_s = \emptyset$ for $r \neq s$. The following proposition shows that whenever the circloids are pairwise disjoint or, in other words, whenever the order is strict $r < s \implies \mathcal{C}_r \prec \mathcal{C}_s$, we can construct a semiconjugacy with $R_{\varrho}$ where each fiber of the semiconjugacy contains the projection of a circloid of the family $\{\mathcal{C}_r\}_{r\in \R}$.

\begin{proposition} \label{prop:construction_of_irrational_factor}
    If the J\"ager's family $\left\{ \mathcal{C}_r \right\}_{r\in \R}$ is pairwise disjoint, then $f$ is a topological extension of an irrational rotation $R_\varrho$.
\end{proposition}

\begin{proof}

 Consider the map $H \colon \A \to \R$ given by \[H(\hat z) := \sup \left\{ r \in \R : \hat z \in \mathcal{U}^+(\mathcal{C}_r) \right\}.\]
This map is well-defined. In fact, write $\hat z = (x, \tilde y)$. It follows from \eqref{eq:localization} that $\hat z \in \mathcal{U}^+(\mathcal{C}_r)$ for any $r < y-\tau$, so the set $\{r \in \R : \mathcal{U}^+(\mathcal{C}_r)\}$ is nonempty. Analogously, $\hat z\in\Um(\sC_r)$ for any $r>\tilde y+\tau$, so this set is also bounded from above. In particular,
\[
|H(\hat z)-\tilde y|\leq\tau\qquad\text{for every }\hat z=(x,\tilde y)\in\A.
\]

\begin{claim}
The map $H$ is continuous and, for every $\hat z\in\A$, it satisfies
\begin{align}
H\circ\Th(\hat z)&=H(\hat z)+1,\label{eq:H-deck}\\
H\circ\fh(\hat z)&=H(\hat z)+\varrho.\label{eq:H-equiv}
\end{align}
\end{claim}

By pairwise disjointness of $\{\sC_r\}_{r\in\R}$, we have $\sC_t\subseteq H^{-1}(t)$ for every $t\in\R$. In fact, if $\hat z\in\sC_t$, it lies above every $\sC_r$ with $r<t$ and below every $\sC_s$ with $s>t$.

To prove continuity, fix $\hat z$ and $\epsilon>0$. Put $a:=H(\hat z)-\epsilon$ and $b:=H(\hat z)+\epsilon$. Choose $r$ with $a<r<H(\hat z)$. The nesting of upper components gives $\hat z\in\Up(\sC_r)$. We also have $\hat z\in\Um(\sC_b)$. Consequently, $V:=\Up(\sC_r)\cap\Um(\sC_b)$ is an open neighborhood of $\hat z$. For every $\hat w\in V$, we have $a<H(\hat w)\leq b$. Thus, $|H(\hat w)-H(\hat z)|\leq\epsilon$.

Now, since $\Th(\hat z)\in\Up(\sC_r)$ if and only if $\hat z\in\Up(\sC_{r-1})$, the set whose supremum defines $H(\Th(\hat z))$ is obtained by adding $1$ to the set defining $H(\hat z)$. This proves \eqref{eq:H-deck}. Similarly, since $\fh$ is end-preserving,
\[
\fh(\hat z)\in\Up(\sC_r)\iff\hat z\in\Up(\sC_{r-\varrho}),
\]
and taking suprema proves \eqref{eq:H-equiv}.

By \eqref{eq:H-deck}, we obtain a well-defined continuous map $h\colon\T^2\to\T$, $h(z):=\pi\circ H(\hat z)$ for every $\hat z\in\widehat\pi^{-1}(z)$. By this last estimate and the connectedness of $\A$, $H$ is onto $\R$, so $h$ is surjective. Finally, since $\fh$ is a lift of $f$ with respect to $\widehat\pi$, equation \eqref{eq:H-equiv} descends to $h\circ f=R_\varrho\circ h$.
\end{proof}

\section{Proof of Theorem~\ref{thm:main}}
\label{sec:proof_main_thm}

By Proposition~\ref{prop:construction_of_irrational_factor}, to prove Theorem~\ref{thm:main} it suffices to show the following result.

\begin{theorem}
Assume that $f\in\Homeo{k}(\T^2)$, $k\neq0$, and that its lift $\ft$ has $\varrho$-bounded vertical deviations for some $\varrho\in\R\setminus\Q$. Then its associated family of Jäger circloids $\{\sC_r\}_{r\in\R}$ is pairwise disjoint.
\end{theorem}

We adapt the planar arguments of Passeggi-Sambarino \cite[\S 2]{PasseggiSambarino2020}, interchanging the coordinates to work on $\T \times \R$. Their dynamical Corollary 2.4 is replaced by Lemma~\ref{lem:pullback} and the shear estimates below.

\subsection{The Passeggi-Sambarino method}

Let $\mathscr{C}^-\preceq\mathscr{C}^+$ be distinct intersecting circloids. Following~\cite{PasseggiSambarino2020},
we say that an annular continuum $A$ is \emph{strongly contained} in the bunch determined
by $\mathscr{C}^-$ and $\mathscr{C}^+$, and write $A\Subset[\mathscr{C}^-,\mathscr{C}^+]$, if
\begin{equation}\label{eq:strong}
A \subset [\mathcal{C}^-, \mathcal{C}^+], \qquad \mathcal{C^-} \not\subset A, \qquad \mc C^+ \not\subset A.
\end{equation}

\begin{remark}\label{rem:order}
In~\cite{PasseggiSambarino2020}, the order is defined by $B \subset \overline{ \mc U^+(A)}$. Proposition~\ref{prop:PS} adapts their proofs to \eqref{eq:strong}, counting the translates of a common-boundary point instead of copies of an intersection component.
\end{remark}

Fix an annular continuum $A \Subset [\mc C^-, \mc C^+]$, and write $\tilde E = \pi^{-1}(E)$ for any $E \subset \A$. Choose $p \in \partial \mc C^- \cap \mc \partial C^+$ and a lift $\tilde p$. This point exists. Indeed, nesting of the complementary components gives $\emptyset \neq \partial(\mc C^- \cap \mc C^+) \subset \partial \mc C^- \cap \partial C^+$. Write $(\tilde x, \tilde y) = \tilde p$. For a planar continuum $Z$, we define  
\begin{displaymath}
    N_p(Z) \coloneq \#\{\ell \in \Z : (\tilde x + \ell, \tilde y) \in Z\}
\end{displaymath}
The following proposition is adapted from the proofs of \cite[Propositions 2.2 and 2.5]{PasseggiSambarino2020}.

\begin{proposition}\label{prop:PS}
    Let $A \Subset [\mc C^-, \mc C^+]$ and fix $p, \tilde p = (\tilde x, \tilde y)$ as above. If $Z_j \subset \tilde A$ are continua with $\diam(Z_j) \to +\infty$, then $N_p(Z_j) \to \infty$. Moreover, any two points of $A$ have lifts contained in a common continuum in $\tilde A$.
\end{proposition}
\begin{proof}
Suppose that the first assertion fails. Passing to a subsequence, assume that $N_p(Z_j)$ is bounded. Since $\widetilde A$ is vertically
bounded, $\diam(\operatorname{pr}_1 Z_j)\to\infty$. Hence, after taking a further subsequence and horizontal integer translates, we obtain continua $W_j\subset\widetilde A$ such that
\begin{displaymath}
(\tilde x, \tilde y)\notin W_j,
\qquad
[\tilde x-j,
 \tilde x+j]
\subset\operatorname{pr}_1(W_j).
\end{displaymath}
Indeed, the number of translations satisfying the second condition tends to infinity, whereas only boundedly many contain $\widetilde p$.

We may replace each $W_j$ by its planar filling, that is, its union with the bounded components of its complement. This filling remains in $\widetilde A$, since every component of $\mathbb R^2\setminus\widetilde A$ is unbounded. It also avoids $\widetilde p$: every neighborhood of $\widetilde p$ meets $\pi^{-1}(\mathcal U^-(\mc C^-))$, whose components are unbounded and disjoint from $W_j$. Here we use $\partial \mc C^-=\partial\mathcal U^-(\mc C^-)$, which follows from \cite[Corollary 3.3]{Jager2009}. The horizontal projection is unchanged by filling.

Attach to $W_j$ two horizontal rays, one extending leftwards from a point of minimal first coordinate and the other extending rightwards from a point of maximal first coordinate. Since $W_j$ is filled, the complement of their union with $W_j$ has exactly two components, one containing the upper end of the plane and the other containing the lower end. The point $\widetilde p$ lies in one of them. Passing to a subsequence, assume that it lies in the upper component for every $j$; the other case is symmetric.

We now apply the separation construction in the proof of \cite[Proposition 2.2]{PasseggiSambarino2020}, with coordinates interchanged. The required reference curves can be chosen in $\pi^{-1}(\mathcal U^-(\mc C^-))$, ending arbitrarily close to $\widetilde p$ in the upper component just described. These curves avoid $W_j$. The construction therefore yields continua
\[ L_j\subset\widetilde C^-, \qquad W'_j\subset W_j,\]
whose horizontal diameters tend to infinity, such that $W'_j$ is $a_j$-virtually below $L_j$, with $a_j\to\infty$, in the terminology of \cite{PasseggiSambarino2020}.

After passing to subsequences, the compactness and filling argument in the proof of \cite[Lemma 2.1]{PasseggiSambarino2020} gives
\begin{displaymath}
\operatorname{Fill}_{\mathbb A}
 \left(\lim_H\pi_{\A}(L_j)\right)= \mc C^-,
\qquad
D:=
\operatorname{Fill}_{\mathbb A}
 \left(\lim_H\pi_{\A}(W'_j)\right)\subset A,    
\end{displaymath}
where $D$ is an annular continuum. The first equality follows from the minimality of $\mc C^-$. The proof of \cite[Lemma 2.3]{PasseggiSambarino2020} yields
\[D\subset \mc C^-\cup\mathcal U^-(\mc C^-).\]
Together with $D\subset A\subset \mc C^-\cup\mathcal U^+(\mc C^-)$, this implies $D\subset \mc C^-$. Minimality gives $D=\mc C^-$, contradicting $\mc C^-\not\subset A$. The symmetric case gives $\mc C^+\subset A$, likewise a contradiction.

For the second assertion, the construction in the proof of \cite[Proposition 2.5]{PasseggiSambarino2020} provides continua of arbitrarily large diameter through prescribed lifts of the two points. By the first assertion, choose one through each lift containing a translate of $\widetilde p$. A horizontal integer translation of one continuum makes these points coincide, and their union proves the claim.
\end{proof}

We shall not use \cite[Corollary 2.4]{PasseggiSambarino2020} exactly as written. The following direct observation is the required replacement.

\begin{lemma}\label{lem:pullback}
Let $Z\subset\R^2$ be a continuum, $X\subset\R^2$ be a nonempty set, $\widehat G\in\Homeo{0}(\A)$ and $G\colon\R^2\to\R^2$ be a lift of $\widehat G$. Then
\begin{equation}\label{eq:pullback-bound}
 \#\bigl\{\ell\in\Z: T_{(\ell,0)}X\subset G(Z)\bigr\}\leq\bigl\lfloor\diam(\pr{1} Z)\bigr\rfloor+1.
\end{equation}
\end{lemma}

\begin{proof}
Let $E\subset\Z$ be the set counted on the left. If $E=\emptyset$ there is nothing to prove; otherwise choose any point $\tilde z\in G^{-1}(X)$. Since $G$ is a lift of a homeomorphism of the annulus, $G\circ T_{(1,0)}=T_{(1,0)}\circ G$. Thus, for each $\ell\in E$,
\[T_{(\ell,0)}(\tilde z)\in G^{-1}\bigl(T_{(\ell,0)}X\bigr)\subset Z.\]
Choose any finite subset $\{\ell_1<\dots<\ell_N\}\subset E$. The first-coordinate distance between $T_{(\ell_1,0)}\tilde z$ and $T_{(\ell_N,0)}\tilde z$ is $\ell_N-\ell_1\geq N-1$. Since both points belong to $Z$,
\[N-1\leq\ell_N-\ell_1\leq\diam(\pr{1} Z).\]
This holds for every finite subset of $E$. Hence $E$ is finite and has at most $\lfloor\diam(\pr{1} Z)\rfloor+1$ elements, proving~\eqref{eq:pullback-bound}.
\end{proof}

\subsection{The shear-growth estimate}

The next computations are the mechanism replacing the two distinct rotation vectors in Passeggi-Sambarino \cite{PasseggiSambarino2020}.

\begin{lemma}\label{lem:qstep}
For any $q\ge1$, $\tilde z=(\tilde x, \tilde y)\in\R^2$ and $z = \pi(\tilde z)$, we can write
\begin{equation}\label{eq:qstep}
 \pr{1}\bigl(\widetilde f^{q}(\tilde z)-\tilde z\bigr)
 =qk\tilde y+\frac{k\varrho q(q-1)}2+R_q(z),
\end{equation}
where
\[ |R_q|\le q\bigl(|k|C+\|\Delta_1\|_\infty\bigr).\]
\end{lemma}
\begin{proof}
Write $(\tilde x_j,\tilde y_j)=\widetilde f^{j}(\tilde z)$. By \eqref{eq:bvd}, $\tilde y_j=\tilde y+j\varrho+D_j$ with $|D_j|\le C$, where $D_j$ depends only on $z$ and $j$. Since $\tilde f = I_k + \Delta_{\tilde f}$,
\[ \tilde x_{j+1}-\tilde x_j=k\tilde y_j+\Delta_1(f^j(z)),\]
summing from $j=0$ to $q-1$ gives \eqref{eq:qstep} with
\[
 R_q=k\sum_{j=0}^{q-1}D_j+
 \sum_{j=0}^{q-1}\Delta_1\circ f^j,
\]
and the bound follows.
\end{proof}

\begin{proposition}\label{prop:shear}
For every $q\ge1$ and $\tilde z\in\R^2$, there is $K_{q,\tilde z}>0$, independent of $n$, such that
\begin{equation}\label{eq:shear}
 \left|\pr{1}\bigl(\widetilde f^{n+q}(\tilde z)-\widetilde f^{n}(\tilde z)\bigr)-knq\varrho\right|
 \le K_{q,\tilde z} \quad \text{for every} \ n\ge0.
\end{equation}
\end{proposition}
\begin{proof}
Apply Lemma \ref{lem:qstep} to $\widetilde f^{n}(\tilde z)$. If $\tilde y=\pr{2}(\tilde z)$, then
\[ \pr{2}(\widetilde f^{n}(\tilde z))=\tilde y+n\varrho+D_n,
 \qquad |D_n|\le C. \]
After subtracting $knq\varrho$, the remaining expression is bounded in absolute value by
\[ K_{q,\tilde z}=q|k|\left(\left|\tilde y+\frac{\varrho(q-1)}2\right|+2C\right)
 +q\|\Delta_1\|_\infty.\]
\end{proof}

\subsection{Disjointness}
Reasoning by contradiction, let us assume that $\mathcal{C}_{r_0}\cap\mathcal{C}_{s_0}\ne\varnothing$ for some $r_0<s_0$. These circloids are distinct by Lemma \ref{lem:distinct}. Choose
\[  r_0<a^-<a<t<b<b^+<s_0 \]
and set
\[ A \coloneq [\mathcal{C}_a,\mathcal{C}_b],\qquad B \coloneq [\mathcal{C}_{a^-},\mathcal{C}_{b^+}].\]
Both $A$ and $B$ are annular continua strongly contained in the bunch determined by $\sC_{r_0}$ and $\sC_{s_0}$, and $A\subset B$. Indeed, by~\ref{eq:weak-order} and the facts listed in Section~\ref{sec:annular_continua_and_circloids}, we have $\sC_{r_0}\preceq A\preceq\sC_{s_0}$ and $A\subset B$. Moreover, if $\mathscr{C}_{r_0}\subset A$, then $\mathscr{C}_{r_0}\subset\mathscr{C}_a\cup\mathscr{U}^+(\mathscr{C}_a)$, that is, $\mathscr{C}_a\preceq\mathscr{C}_{r_0}$; together with $\mathscr{C}_{r_0}\preceq\mathscr{C}_a$ and antisymmetry, this gives $\mathscr{C}_{r_0}=\mathscr{C}_a$, contradicting Lemma~\ref{lem:distinct}. Similarly, if $\mathscr{C}_{s_0}\subset A$, then $\mathscr{C}_{s_0}\subset\mathscr{C}_b\cup\mathscr{U}^-(\mathscr{C}_b)$, that is, $\mathscr{C}_{s_0}\preceq\mathscr{C}_b$, and we conclude in the same way. The same argument applies to $B$.

Choose $0 < \eta<\min\{t-a,b-t\}$. Since $\varrho$ is irrational, there are $q\ge1$ and $m\in\Z$ such that $0<|\delta|<\eta$, where $\delta := q\varrho-m$. Take any $\widehat z\in\mathcal{C}_t$ and consider
\[\widehat w:=\widehat T_1^{-m}\widehat f^{q}(\widehat z)\in\mathcal{C}_{t+\delta}.\]
Both points belong to $A$. By the second part of Proposition \ref{prop:PS}, there is a continuum $Z\subset\widetilde A$ containing lifts $\tilde z\in\pi_\A^{-1}(\widehat z)$ and $\tilde w\in\pi_\A^{-1}(\widehat w)$.

For some $\ell\in\Z$, we can write
\[ \tilde w=T_{(\ell,-m)}\widetilde f^{q}(\tilde z).\]
Since $\ft^{\,n}\circ T_{(\ell,-m)}=T_{(\ell-knm,-m)}\circ\ft^{\,n}$, we get
\[ \widetilde f^{n}(\tilde w)=T_{(\ell-nkm,-m)}\widetilde f^{n+q}(\tilde z).\]
Together with \eqref{eq:shear}, this gives
\begin{equation}\label{eq:growth}
 \left|\pr{1}\bigl(\widetilde f^{n}(\tilde w)-\widetilde f^{n}(\tilde z)\bigr)-nk\delta\right|
 \le |\ell|+K_{q,\tilde z}.
\end{equation}

Choose $n_j\to\infty$ and $m_j\in\Z$ such that $\delta_j \coloneq n_j\varrho-m_j\longrightarrow0$, and define $G_j \coloneq T_{(0,-m_j)}\widetilde f^{n_j}$. It induces the annulus homeomorphism $\widehat G_j=\Th^{\,-m_j}\fh^{\,n_j}\in\Homeo{0}(\A)$, and
\[
 \widehat G_j(A)=[\mathcal{C}_{a+\delta_j},\mathcal{C}_{b+\delta_j}]\subset B
\]
for all sufficiently large $j$, namely as soon as $|\delta_j|<\min\{a-a^-,b^+-b\}$. Hence $G_j(Z)\subset\widetilde B$. On the other hand, since $T_{(0,-m_j)}$ does not change first coordinates, \eqref{eq:growth} yields
\[
 \diam(G_j(Z))\ge \diam(\pr{1}G_j(Z))
 \ge |k\delta|n_j-|\ell|-K_{q,\tilde z}\longrightarrow\infty,
\]
where we used that $k \neq 0$ and $\delta \neq 0$.

Fix $p \in \partial \mc C_{r_0} \cap \partial \mc C_{s_0}$ and a lift $\tilde p = (\tilde x, \tilde y)$. Since $B \Subset [\mc C_{r_0}, \mc C_{s_0}]$, Proposition~\ref{prop:PS} gives \[N_p(G_j(Z)) \to +\infty.\]
On the other hand, Lemma~\ref{lem:pullback}, applied with $X = \{\tilde p\}$, yields \[N_p(G_j(Z)) \leq \floor{\diam (\pr{1} Z)} + 1 \quad \text{for every} \ j,\]
a contradiction. Thus the family is pairwise disjoint and Theorem~\ref{thm:main} follows.

\section{Proof of Theorem~\ref{thm:main-semi-conjugacy-uniqueness}}
\label{sec:proof-main-semi-conjugacy-uniqueness}

First of all, observe that if $\tilde f$ exhibits $\varrho$-bounded vertical deviations, $(h_i,\varrho)$ is an irrational factor of $f$, and $\tilde h_i\colon\R^2\to\R$ is a lift of $h_i$, for $i=1,2$, then it holds
\begin{displaymath}
  \tilde h_i\circ \tilde f(\tilde x,\tilde y) = \tilde h_i(\tilde x,\tilde y) + \varrho,
\end{displaymath}
and
\begin{displaymath}
  \tilde h_i(\tilde x+m,\tilde y + n) = \tilde h_i(\tilde x, \tilde y) + n, \quad \text{for every } (\tilde x, \tilde y) \in \R^2,\ (m,n) \in \Z^2.
\end{displaymath}
This clearly implies that the function $\phi\colon\R^2\to\R$ given by
\begin{displaymath}
  \phi(\tilde x, \tilde y) \coloneq \tilde h_2(\tilde x, \tilde y) - \tilde h_1(\tilde x, \tilde y), \quad \text{for every } (\tilde x, \tilde y) \in \R^2,
\end{displaymath}
is continuous, $\Z^2$-periodic and $f$-invariant.

We want to show that $\phi$ is constant on $\T^2$. Reasoning by contradiction, let us suppose that $\min_{z}\phi(z)<\max_z\phi(z)$. Let $c$ be any number with $\min_{z}\phi(z)<c<\max_z\phi(z)$. Then, both sets $\{\phi > c\}$ and $\{\phi < c\}$ are open, nonempty, and, by Lemma~\ref{lemma:U-c-cc-periodic}, each connected component of them is periodic. Moreover, we shall prove the following 

\begin{claim}\label{claim:connect_pre_image_phi}
Every connected component of $\{\phi<c\}$ and of $\{\phi>c\}$ is inessential, and every connected component of its preimage under $\pi$ is unbounded in the vertical direction.
\end{claim} 

\begin{proof}
    Let $V$ be a connected component of either $\{\phi<c\}$ or $\{\phi>c\}$, and let $\widetilde V$ be a connected component of $\pi^{-1}(V)$. We first show that $\widetilde V$ is unbounded in the vertical direction. Suppose, to the contrary, that there exists $C>0$ such that $\widetilde V\subset\R\times[-C,C]$. By the periodicity of $V$, there exist $n\geq1$ and $\boldsymbol{q}=(q_1,q_2)\in\Z^2$ such that
\[\widetilde f^n(\widetilde V)=T_{\boldsymbol{q}}(\widetilde V) \subset \R \times [q_2-C, q_2+C].\]
Since $\widetilde f$ is a lift of a homeomorphism
homotopic to $I_k$, we have, for each $m \in \N$,
\begin{displaymath}
    \tilde f^{mn}(\tilde V) = T_{Q_m}(\tilde V) \quad \text{where}\; \; Q_m \coloneq \sum_{j=0}^{m-1}I_{kn}^j\boldsymbol{q}.
\end{displaymath}
Since $\pr{2}(Q_m) = mq_2$, we conclude that
\begin{displaymath}
  \tilde f^{mn}(\tilde V) \subset\R\times[mq_2-C,mq_2+C].
\end{displaymath}
This clearly implies that $\varrho=\frac{q_2}{n}$, with $n\in\N$ and $q_2\in\Z$, contradicting the fact that $\varrho$ is irrational.

We next show that \(V\) is inessential. Suppose otherwise. We claim that $V$ contains a simple closed curve representing the horizontal homology class $(1,0)$. Indeed, if \(V\) is annular, choose a simple essential closed curve \(\gamma\subset V\), and let \((a,b)\in\mathbb Z^2\) be its homology class. Since \(f^n(V)=V\), the curve \(f^n\circ\gamma\) also lies in \(V\), so its homology class \((a+knb,b)\) is collinear with \((a,b)\). Their determinant is \(-knb^2\), and therefore \(b=0\). Thus \(V\) contains a simple horizontal essential closed curve. If \(V\) is fully essential, it also contains such a curve.

Suppose first that \(V\) is a component of \(\{\phi<c\}\), and choose a component \(W\) of the nonempty set \(\{\phi>c\}\). Let \(\gamma\subset V\) be a simple horizontal essential closed curve, and let \(L\) be a connected component of \(\pi^{-1}(\gamma)\). The lines \(T_{(0,j)}(L)\), \(j\in\mathbb Z\), are pairwise disjoint and divide the plane into vertically bounded regions between consecutive lines. Since \(W\cap\gamma=\varnothing\), every connected component \(\widetilde W\) of \(\pi^{-1}(W)\) is disjoint from these lines. By connectedness, \(\widetilde W\) is contained in one of those regions and is therefore vertically bounded, contradicting the first part of the claim.

If \(V\) is a component of \(\{\phi>c\}\), the same argument applies after exchanging the two strict level sets. Hence every component of either set is inessential.
\end{proof}

\begin{proof}[Proof of Theorem~\ref{thm:main-semi-conjugacy-uniqueness}]
    
Let $z_0\in\mathbb{T}^2$ be such that $\phi(z_0)=\min\phi$, and put $ c' := \tfrac{1}{2}\left(\min\phi + c\right)$. Let 
\begin{displaymath}
    V \coloneq \cc\big(\{\phi < c\}, z_0\big) \quad \text{and} \quad V' \coloneq \cc\big(\{\phi < c'\}, z_0\big).
\end{displaymath}

Since $V'$ is a connected subset of $\{\phi<c\}$ containing $z_0$, we have $V'\subset V$. Let $\tilde V$ be a connected component of $\pi^{-1}(V)$. Since $\pi(\tilde V)=V$, there exists $\tilde z_0\in\tilde V$ with $\pi(\tilde z_0)=z_0$. Let $\tilde V'$ be the connected component of $\pi^{-1}(V')$ containing $\tilde z_0$. It is a connected subset of $\pi^{-1}(V)$ meeting $\tilde V$, hence $\tilde V'\subset\tilde V$. By Claim~\ref{claim:connect_pre_image_phi}, applied both to $c$ and to $c'$ (which lie strictly between $\min\phi$ and $\max\phi$), $V$ is inessential and $\tilde V'$ is unbounded in the vertical direction.
 
The translates $T_{\boldsymbol{n}}(\tilde V)$, $\boldsymbol{n} \in \Z^2$ are connected components of $\pi^{-1}(V)$, so two of them either coincide or are disjoint. Suppose $T_{\boldsymbol{n}}(\tilde V)=\tilde V$ for some $\boldsymbol{n}\neq 0$. Choosing $\tilde z\in\tilde V$ and a path in $\tilde V$ joining $\tilde z$ to $T_{\boldsymbol{n}}(\tilde z)$, its projection is a closed curve in $V$ whose homology class is $\boldsymbol{n}\neq 0$, contradicting the inessentiality of $V$. Thus, \[T_{\boldsymbol{n}}(\tilde V) \cap \tilde V = \emptyset, \quad \text{for every} \ \boldsymbol{n} \in \Z^2\setminus\{0\}.\]
 
Now, note that $\partial T_{\boldsymbol n}(\tilde V) \subset \{\phi = c\}$ for every $\boldsymbol n\in\mathbb{Z}^2$. Indeed, let $\tilde w\in\partial T_{\boldsymbol n}(\tilde V)$. If $\pi(\tilde w)\in V$, then $\tilde w$ belongs to some connected component of $\pi^{-1}(V)$; this component is open and, since $\tilde w$ is adherent to $\Tn(\tilde V)$, it meets $\Tn(\tilde V)$, hence coincides with it. Thus $\tilde w\in \Tn(\tilde V)$, which is absurd because $\Tn(\tilde V)$ is open. Therefore $\pi(\tilde w)\in\overline{V}\setminus V$. If $\tilde w \in \{\phi < c\}$, then $\pi(\tilde w)$ would belong to an (open) connected component of $\{\phi<c\}$ meeting $V$, that is, to $V$ itself, which we have just excluded. Hence $\tilde w \in \{ \phi \geq c\}$. On the other hand, $\tilde w$ is a limit of points of $\Tn(\tilde V)\subset\{\phi<c\}$, so $\phi(\tilde w)\leq c$ by continuity. Thus $\phi(\tilde w)=c$.
 
Since $\tilde V'$ is connected and vertically unbounded,
$\operatorname{pr}_2(\tilde V')$ is an unbounded interval. Let us assume it
contains $[j_0,+\infty)$ for some $j_0\in\mathbb{Z}$; the case where it is
unbounded from below is analogous. For each integer $j\geq j_0$ choose
$(x_j,j)\in\tilde V'$ and set
$\boldsymbol{n}_j:=\left(-\lfloor x_j\rfloor,-j\right)\in\mathbb{Z}^2$. Then
\begin{displaymath}
  T_{\boldsymbol n_j}(x_j,j) = \left(x_j-\lfloor x_j\rfloor,0\right)\in I,   
\end{displaymath}
where $I \coloneq [0,1] \times \{0\}$. So $T_{\boldsymbol n_j}(\tilde V')\cap I\neq\emptyset$. The vectors $\boldsymbol n_j$ are pairwise distinct, since their second coordinates are. Hence the set
\begin{displaymath}
  \mathcal{A}:=\left\{n\in\mathbb{Z}^2 :
  T_{\boldsymbol n}(\tilde V')\cap I\neq\emptyset\right\}  
\end{displaymath}
is infinite.

For each $\boldsymbol n\in\mathcal{A}$ choose $p_{\boldsymbol n}\in \Tn(\tilde V')\cap I$. By the $\mathbb{Z}^2$-periodicity of $\phi$, we have $\phi(p_{\boldsymbol n})<c'$. Since $\Tn(\tilde V')\subset \Tn(\tilde V)$ and the latter sets are pairwise disjoint, the points $p_{\boldsymbol n}$ are pairwise distinct. Therefore $\{p_{\boldsymbol n}\}_{\boldsymbol n\in\mathcal{A}}$ is an infinite subset of the compact segment $I$, so it has an accumulation point. Consequently, for every $\varepsilon>0$ there exist $\boldsymbol n\neq \boldsymbol m$ in $\mathcal{A}$ such that $|p_{\boldsymbol n}-p_{\boldsymbol m}|<\varepsilon$.
 
Consider the segment $[p_{\boldsymbol n},p_{\boldsymbol m}]\subset I$. The set
$[p_{\boldsymbol n},p_{\boldsymbol m}] \cap \Tn(\tilde V)$ is nonempty (it contains $p_{\boldsymbol n}$) and relatively open, and it is not the whole segment, because $p_{\boldsymbol m}\in T_{\boldsymbol m}(\tilde V)$, which is disjoint from $\Tn(\tilde V)$. By connectedness of the segment, it is not relatively closed, so there exists
$q\in[p_{\boldsymbol n},p_{\boldsymbol m}]$ with $q\in\partial \Tn(\tilde V)$. So 
$\phi(q)=c$, and thus
\begin{displaymath}
  |q-p_{\boldsymbol n}|\leq |p_{\boldsymbol m}-p_{\boldsymbol n}|<\varepsilon
  \qquad\text{and}\qquad
  \phi(q)-\phi(p_{\boldsymbol n})>c-c'>0 .   
\end{displaymath}

Since $\varepsilon>0$ is arbitrary and $c-c'$ does not depend on $\varepsilon$, this contradicts the uniform continuity of $\phi$ on $\mathbb{R}^2$.
 
\medskip
The assumption $\min\phi<\max\phi$ is therefore absurd, so $\phi$ is constant, say $\phi\equiv\tilde\alpha$. Then
$\tilde h_1=\tilde h_2+\tilde\alpha$ on $\mathbb{R}^2$ and, projecting,
\begin{displaymath}
      h_1(z)=h_2(z)+\alpha \quad\text{for every } z\in\mathbb{T}^2.
\end{displaymath}
This completes the proof of Theorem~\ref{thm:main-semi-conjugacy-uniqueness}.
 \end{proof}

\section{Proof of Theorem~\ref{thm:main-wandering-set}} \label{sec:proof-wandering-set}
The proof of Theorem~\ref{thm:main-wandering-set} combines three ingredients. First, we show that the non-wandering set of $f$ is weakly transitive so that the coboundary obstruction of Lemma~\ref{lemma:coboundary_obstruction} becomes available; this is where we use that non-wandering points are fully essential. Second, we construct local perturbations inside wandering discs. Third, we apply Theorem~\ref{thm:main} to the perturbed map and compare the two factors.

We start with the following lemma.

\begin{lemma}\label{lemma:our_f_is_weaklytransitive}
If $f \in \Homeo{k}(\T^2)$, $k \neq 0$, satisfies the hypotheses of Theorem~\ref{thm:main-wandering-set}, then $\Omega(f)$ is weakly transitive.
\end{lemma}

\begin{proof}
    Let $A, B$ be open sets intersecting $\Omega(f)$. Choose $x \in A \cap \Omega(f)$ and $y \in B \cap \Omega(f)$, and shrink $A, B$ to connected open neighborhoods of $x,y$. It follows from Proposition~\ref{prop:kocsard_eventuallyannular} and Lemma~\ref{lemma:our_f_everypoint_is_fullyessential} that both $\mathcal{U}_f(A)$ and $\mathcal{U}_f(B)$ are fully essential.

    We claim that $A^+ \coloneq \bigcup_{j \geq 1}f^j(A)$ is also fully essential. In fact, choose two cycles $\gamma_1, \gamma_2 \colon \T \to \mathcal{U}_f(A)$ whose homology classes generate $H_1(\T^2, \Z)$. Since $K_A \coloneq \gamma_1(\T) \cup \gamma_2(\T)$ is compact and $\mathcal{U}_f(A)$ is covered by the open sets $f^j(A)$, $j \in \Z$, there exist $j_1 < j_2 < \ldots < j_r$ such that $\{f^{j_1}(A), \ldots, f^{j_r}(A)\}$ covers  $K_A$. Choose $N \in \N$ so large that $N + j_1 \geq 1$. Then $f^N(K_A) \subset \bigcup_i f^{N+j_i}(A) \subset A^+$, and since $f^N$ acts on $H_1(\T^2, \Z)$ by the automorphisms $I_{kN}$, the homology classes of $f^N \circ \gamma_1$ and $f^N \circ \gamma_2$ still generate $H_1(\T^2, \Z)$. Thus $A^+$ is fully essential.

    Applying the same argument to two closed curves in  $\mathcal{U}_f(B)$, this time shifting the finitely many indices far enough in the negative direction, we get $B^- \coloneq \bigcup_{j\geq 1} f^{-j}(B)$ is fully essential.

    Two fully essential open subsets of $\T^2$ always intersect: one contains a closed curve
of homology class $(1,0)$, the other one of class $(0,1)$, and these two curves have algebraic intersection number $\pm 1$. Thus $A^+\cap B^- \neq \emptyset$, that is, $f^i(A) \cap f^{-j}(B) \neq \emptyset$ for some $i,j \geq 1$. Therefore $f^n(A) \cap B \neq \emptyset$ with $n \coloneq i+j \geq 1$.
\end{proof}

\subsection{The local perturbation}

In this subsection we show that any non-constant continuous function can be displaced by a compactly supported isotopy in such a way that the resulting variation $\delta$ violates the symmetry $\max \delta = -\min \delta$ imposed by Lemma~\ref{lemma:coboundary_obstruction}.

\begin{lemma}\label{lemma:the_local_pertubation}
Let $U \subset \R^2$ be a nonempty connected open set, and let $u\in C^0(U,\R)$ be non-constant. Then there exist a compact set $K \subset U$ and a smooth isotopy $(\ell_\tau)_{{\tau \in [0,1]}}$ of diffeomorphisms of $\R^2$ with the following properties:
\begin{itemize}
    \item $\ell_0 = \mathrm{id}$ and $\ell_{\tau}$ is the identity on $\R^2\setminus K$ for every $\tau \in [0,1]$. In particular, $\ell \coloneq \ell_1$ satisfies $\ell(K) = K$ and $\ell(U) = U$.
    \item The function $\delta \colon \R^2 \to \R$ defined by \[\delta(z) = \begin{cases}
        u(\ell(z)) - u(z), &\quad z \in U, \\
        0, &\quad z \in \R^2 \setminus U
    \end{cases}\]
    is continuous, vanishes outside $K$, and satisfies $ \max \delta > - \min \delta$.
\end{itemize}
\end{lemma}

Before giving the proof, let us explain the meaning of the lemma and the idea behind its proof. The function $\delta$ measures the change in the value of $u$ when the points of $U$ are displaced by $\ell$. Since $\ell$ is the identity outside a compact subset of $U$, the function $\delta$ vanishes on a nonempty open set, so that $\min\delta\le 0\le\max\delta$.

The content of the lemma is a quantitative asymmetry between these two numbers: the displacement $\ell$ produces at the point $z_0$ a gain $a$ which is strictly larger than the largest loss suffered at any other point. This is precisely the situation excluded by the coboundary obstruction of Lemma~\ref{lemma:coboundary_obstruction}.

To construct $\ell$, we select a short segment along which $u$ increases from its minimum value on the segment to its maximum value on the segment, without attaining either of these values in the interior. We then push the points of a thin neighbourhood of this segment in the direction of increasing $u$, in such a way that the initial endpoint is carried exactly onto the final one. The initial endpoint therefore gains the full oscillation $a$ of $u$ along the segment. Every other point is moved forward along the same line and, thanks to the choice of the segment, can lose strictly less than $a$. The remaining steps of the proof are continuity and compactness arguments that transfer this one-dimensional estimate to a two-dimensional neighborhood.

\begin{proof}[Proof of Lemma~\ref{lemma:the_local_pertubation}]  Since $U$ is connected and $u$ is not constant, $u$ is not locally constant, so there exists $x_1 \in U$ and a disc $B \subset U$ centered at $x_1$ such that $u$ is not constant on $B$. Thus there exist points $p, q \in B$ with $u(p) \neq u(q)$. Consider the segment $c(t) \coloneq (1-t)p+tq$, $t \in [0,1]$, which is contained in $B$. The continuous function $t \mapsto u_c(t) \coloneq u(c(t))$ is not constant. Put $m = \min u_c$ and $M = \max u_c$. Choose $t_{-}, t_+ \in [0,1]$ with $u_c(t_{-}) = m$ and $u_c(t_+) = M$. Interchanging $p$ and $q$ if necessary, we may assume that $t_{-} < t_{+}$. Define \[t_1 \coloneq \min\{t \in [t_{-}, 1]: u_c(t) = M\} \ \ \text{and} \ \ t_0 \coloneq \max\{t \in [0, t_1] : u_c(t) = m\}.\]
By definition of these numbers, $t_0 < t_1$ and $u_c(t) \in (m,M)$ for every $t \in (t_0,t_1)$. Set $z_0 \coloneq c(t_0)$ and $z_1 \coloneq c(t_1)$. These are distinct points of $U$, the closed segment $[z_0,z_1]$ is contained in $B \subset U$ and, along $[z_0,z_1]$, the function $u$ attains its minimum $m$ only at $z_0$ and its maximum $M$ only at $z_1$.

After an affine change of coordinates by a similarity of $\R^2$, which does not affect the statement, we may assume that $z_0 = (0,0)$ and $z_1 = (1,0)$. So the segment $[z_0,z_1]$ is $[0,1] \times \{0\} \subset U$ and the function \begin{equation}
    v(x) \coloneq u(x,0), \quad x \in [0,1]
\end{equation} satisfies $v(0) = m$, $v(1) = M$ and
\begin{equation}\label{eq:extreme_in_extreme}
     v(0) < v(x) < v(1) \quad \text{for every} \  x \in (0,1).
\end{equation}
For $x < x'$, the quantity $v(x) - v(x')$ is the loss of $u$ when a point of the segment is moved forward from $x$ to $x'$. Let
\begin{equation}\label{eq:maximum_loss_of_u}
    D \coloneq \max\{v(x)-v(x') : 0 \leq x \leq x' \leq 1\}.
\end{equation}

Since $v(x) - v(x') \leq M-m$, with equality only if $x = 1$ and $x' = 0$, it follows from \eqref{eq:extreme_in_extreme} that \[0\leq D < a \coloneq v(1) - v(0) = M-m.\]

Let $\eta_0, \kappa_0 > 0$ such that $[-\eta_0, 1+\eta_0] \times [-\kappa_0, \kappa_0] \subset U$. We extend $v$ by the same formula $v(x) = u(x,0)$ to $x \in [-\eta_0, 1+\eta_0]$ and, for each $\eta \in [0,\eta_0]$, we set
\[D_\eta \coloneq \max \{v(x) - v(x') : -\eta \leq x \leq x' \leq 1+\eta\}.\]

We claim that there exists $\eta \in (0,\eta_0]$ such that $D_\eta < a$. In fact, if $D_\eta \geq a$ for every $\eta \in (0, \eta_0]$, then for each $j \geq 1$ there exist $x_j \leq x'_j$ in $[-\eta_0/j, 1+\eta_0/j]$ with $v(x_j) - v(x_j') \geq a$. By compactness and continuity, a subsequence of $(x_j, x'_j)$ converges to some $(x,x')$ with $0 \leq x \leq x' \leq 1$ and $v(x)-v(x') \geq a$,  contradicting $D < a$.

Now, put $\epsilon \coloneq (a-D_\eta)/3 > 0$, so that
\begin{equation}\label{eq:D_eta_and_epsilon}
    D_\eta + 2 \epsilon < a.
\end{equation}
The function $u$ is uniformly continuous on the compact rectangle $[-\eta, 1+\eta] \times [-\kappa_0, \kappa_0]$. Consequently, there exists $\kappa \in (0, \kappa_0]$ such that, for the rectangle \[Q \coloneq [-\eta, 1+\eta] \times [-\kappa, \kappa] \subset U,\]
we have
\begin{equation}\label{eq:extension_to_rectangle}
    \abs{u(x,y) - v(x)} < \epsilon \quad \text{for every} \ (x,y) \in Q.
\end{equation}
That is, on the thin rectangle $Q$, the function $u$ differs from its values on the central segment by less than $\epsilon$.

Choose a smooth function $\chi\colon \R \to [0,\infty)$ such that $\mathrm{supp} \ \chi \subset (-\eta, 1+\eta)$ and $\chi > 0$ on $[0,1]$. Choose also a smooth function $\beta \colon \R \to [0, \infty)$ satisfying $\mathrm{supp}\ \beta \subset (-\kappa, \kappa)$ and $\beta(0) = 1$. Consider now the following horizontal vector field on $\R^2$: \[F(x,y) \coloneq \chi(x)\beta(y) \frac{\partial }{\partial x}. \]
It is smooth, and its support is contained in the compact set $K \coloneq \mathrm{supp} \ \chi \times \mathrm{supp} \ \beta$, which lies in the interior of $Q \subset U$. Let $(\Phi^\tau)_{\tau \in \R}$ denote its flow, which is complete since $F$ has compact support. We record three basic facts of this flow:
\begin{enumerate}
    \item $\Phi^\tau$ is the identity outside $K$, and $\Phi^\tau(K) = K$.
    \item Since $\chi \beta \geq 0$, points in $K$ move forward along horizontal lines. That is, $\Phi^\tau(x,y) = (x',y)$ with $x \leq x'$ for $\tau \geq 0$.
    \item The initial endpoint reaches the final endpoint in finite time. On the line $y=0$ we have $\beta(0) = 1$, so the flow equation becomes $\Dot{x} = \chi(x)$. Since $\chi > 0$ on $[0,1]$, we have $\Phi^{\tau_0}(0,0) = (1,0)$, where \[\tau_0 = \int_{0}^1 \frac{d\tau}{\chi(\tau)}.\]
\end{enumerate}
Finally, we define $\ell_\tau \coloneq \Phi^{\tau_0 \tau}$ for $\tau \in [0,1]$ and $\ell \coloneq \ell_1$. This is the required isotopy. The function $\delta$ is well defined, continuous and vanishes outside $K$.

We shall prove that $\max \delta > -\min \delta$. Let $z_0 = (0,0)$. Then, \[\delta(z_0) = u(1,0)-u(0,0) =v(1)-v(0) = a,\]
so $\max \delta \geq a$. Now, let $z = (x,y) \in Q$. Then $\ell(z) = (x',y)$ with $-\eta \leq x \leq x' \leq 1+\eta$. Using \eqref{eq:extension_to_rectangle} twice and then the definition of $D_\eta$, we obtain
\begin{align*}
    \delta(z) &= u(x',y)-u(x,y) > (v(x')-\epsilon) -(v(x)+\epsilon) =  -(v(x)-v(x')) - 2\epsilon \geq -D_\eta-2\epsilon.
\end{align*}
Since $\delta$ vanishes outside $K \subset Q$ and $D_\eta \geq 0$, this gives \[\delta(z) \geq -(D_\eta +2\epsilon) \quad \text{for every} \ z \in \R^2.\]
Thus, it follows from \eqref{eq:D_eta_and_epsilon} that $\min \delta \geq -(D_\eta+2\epsilon) > -a \geq - \max \delta$.
\end{proof}

\begin{remark}\label{rmk:pertube_the_torus}
Lemma~\ref{lemma:the_local_pertubation} will be applied to a connected component $\tilde U$ of $\pi^{-1}(U)$, where $U \subset \T^2$ is an open disc evenly covered by $\pi = \pi_{\T^2}$. Since the isotopy $\ell_\tau$ is the identity outside the compact set $K \subset \tilde U$ and $\pi|_{\tilde U}$ is injective, it descends to a smooth isotopy of $\T^2$ supported in the compact set $\pi(K) \subset U$. This is the form in which the lemma will be used in the proof of Proposition~\ref{prop:h_is_constant}.
\end{remark}

\subsubsection{Perturbing $f$ inside a wandering disk}
Toward the proof of Theorem~\ref{thm:main-wandering-set}, we shall perturb $f$ inside one wandering disc. Assume that $f \in \Homeo{k}(\T^2), k \neq 0,$ satisfies the hypotheses of Theorem~\ref{thm:main-wandering-set} and that $(h, \rho)$ is an irrational circle factor. Let $U$ be a wandering disc for $f$ which is evenly covered by $\pi$, and let $\ell \colon \T^2 \carr$ be a smooth diffeomorphism isotopic to the identity through an isotopy supported in
a compact subset of $U$; in particular, $\ell(U) = U$. Let $g \coloneq f \circ \ell$. Since $\ell$ is isotopic to the identity, $g \in \Homeo{k}(\T^2)$.

\begin{lemma}\label{lemma:the_pertubed_also_wandering}
     $g^n(U) = f^n(U)$ for every $n \in \Z$. In particular, $U$ is a wandering disc for $g$.
\end{lemma}

\begin{proof}
   Recall that $f^n(U) \cap U = \emptyset$ for every $n \neq 0$, and that $\ell$ and $\ell^{-1}$ are the identity outside $U$. For $n = 1$, $g(U) = f(\ell(U))= f(U)$. If $g^n(U) = f^n(U)$ for some $n \geq 1$, then $\ell$ is the identity on $f^n(U)$, so $g^{n+1}(U) = f(\ell(f^n(U)))=f^{n+1}(U)$. Induction gives the formula for all $n \geq 1$. The same induction applied to $g^{-1} = \ell^{-1}\circ f^{-1}$ gives the formula for $n \leq -1$.
\end{proof}

Let $\tilde h\colon \R^2 \to \R$ be the lift of $h$ and fix one connected component $\tilde U \subset \pi^{-1}(U)$. Let $\tilde \ell \colon \R^2 \carr$ be the unique lift of $\ell$ that is the identity outside $\pi^{-1}(U)$; it commutes with the deck transformations. Then define
 \[\delta(z) \coloneq \tilde h\circ \tilde \ell(\tilde z) - \tilde h(\tilde z),\]
where $\tilde z \in \pi^{-1}(z)$.  This is well defined, since $\tilde h(T_{(n,m)}\tilde z) = \tilde h(\tilde z) + m$ for every $(n,m) \in \Z^2$; it is continuous and, since $\ell$ is supported in $U$, $\delta$ vanishes outside $U$.

\begin{lemma}\label{deviation_estmate_of_ppertubed} Let $\tilde g \coloneq \tilde{f} \circ \tilde \ell$, which is a lift of $g \coloneq f \circ \ell$. Then
\begin{equation}\label{eq:deviation_of_g}
    \abs{\pr{2}(\tilde g^n(\tilde z)-\tilde z)-n\varrho} \leq 2B + \|\delta\|_{\infty}, \quad \text{for every} \ \tilde z \in \R^2, \ n \in \Z,
\end{equation}
where $B = \|\tilde h - \pr{2}\|_{\infty}$. In particular, $\tilde g$ has $\varrho$-bounded vertical deviations.
\end{lemma}

\begin{proof}
First, note that using $\tilde g = \tilde f \circ \tilde \ell$ and after adding and subtracting $\tilde h \circ \tilde \ell$, we have
\begin{equation}\label{eq:h_g_less_h}
    \tilde h\circ \tilde g(\tilde z) - \tilde h(\tilde z) = \varrho + \delta(z),
\end{equation}
where $\tilde z \in \R^2$ and $z = \pi(\tilde z)$. For $n \geq 1$, telescoping yields
\begin{equation}\label{eq:telescoping_hg}
    \tilde h \circ \tilde g^n(\tilde z)-\tilde h(\tilde z) - n \varrho = \sum_{j=0}^{n-1}\delta \circ g^j(z).
\end{equation}
By Lemma \ref{lemma:the_pertubed_also_wandering}, $U$ is wandering for $g$, so $\mc O_g(z)$ meets $U$ at most once. Hence the absolute value of \eqref{eq:telescoping_hg} is at most $\|\delta\|_\infty$. Applying the same telescoping identity from $\tilde g^{-n}(\tilde z)$ to $\tilde z$, we obtain this bound for every $n \in \Z$.

Thus, writing $E = \tilde h - \pr{2}$ and using $B = \|E\|_\infty$, we obtain
\begin{align*}
    \abs{\pr{2}(\tilde g^n(\tilde z)-\tilde z)-n\varrho}
    & \leq \abs{\tilde h\circ \tilde g^n(\tilde z) - \tilde h(\tilde z) - n\varrho} + \abs{E(\tilde g^n(\tilde z))} + \abs{E(\tilde z)} \\
    &\leq \|\delta\|_{\infty} + 2B.
\end{align*}
\end{proof}

\subsection{The factor is constant on every wandering domain}
The following is the core of
the proof of Theorem~\ref{thm:main-wandering-set}. Recall that a map $\psi\colon \T^2 \to \T$ is \emph{locally constant} on a set $W \subset \T^2$ if every point of $W$ has a neighborhood on which $\psi$ is constant.

\begin{proposition}\label{prop:h_is_constant} Assume that $f \in \Homeo{k}(\T^2), k \neq 0,$ satisfies the hypotheses of Theorem~\ref{thm:main-wandering-set} and that $(h, \rho)$ is an irrational circle factor. Then  $h$ is locally constant on $\mathcal{W}(f)$.
\end{proposition}

\begin{proof}
    Assume that $h$ is not locally constant at some point $z \in \mathcal{W}(f)$. Since $z$ is wandering, it has a neighborhood $V$ which is a wandering set for $f$. Since $h$ is not constant on any neighborhood of $z$, we can choose an open disc $U \subset V$ containing $z$, evenly covered by $\pi$, such that $h$ is not constant on $U$. Then $U$ is a wandering disc for $f$.

    Choose a connected component $\tilde U$ of $\pi^{-1}(U)$. Then the restriction $\tilde h|_{\tilde U}$ is a non-constant real-valued function. Applying Lemma~\ref{lemma:the_local_pertubation} to $\tilde U$ and $u \coloneq \tilde h|_{\tilde U}$, and using Remark~\ref{rmk:pertube_the_torus}, we obtain a smooth diffeomorphism $\ell \colon \T^2 \carr$, whose lift $\tilde \ell$ that is the identity outside $\pi^{-1}(U)$ coincides on $\tilde U$ with the diffeomorphism given by the lemma. The corresponding function $\delta \colon \T^2 \to \R$ defined above then satisfies \begin{equation}\label{eq1:proof_constance}\max \delta > -\min \delta.
    \end{equation}

Set $g \coloneq f \circ \ell$ and let $\tilde g \coloneq \tilde f \circ \tilde \ell$ be as in Lemma~\ref{deviation_estmate_of_ppertubed}. Then $g \in \Homeo{k}(\T^2)$ and, by Lemma~\ref{deviation_estmate_of_ppertubed} $\tilde g$ has $\varrho$-bounded vertical deviations. Applying Theorem~\ref{thm:main} to $g$, let $h_g$ be the factor of $g$ onto $R_{\varrho}$ constructed in Proposition~\ref{prop:construction_of_irrational_factor}, and let $\tilde h_g$ be a lift, so that $\tilde h_g \circ \tilde g = \tilde h_g + \varrho$ and $\tilde h_g - \pr{2}$ is $\Z^2$-periodic. The function $\tilde h - \tilde h_g$ is also $\Z^2$-periodic. Thus it descends to a continuous function
 \[u \colon \T^2 \to \R, \qquad u(z) \coloneq \tilde h(\tilde z) - \tilde h_g(\tilde z),\] where $\tilde z \in \pi^{-1}(z)$. Now, using \eqref{eq:h_g_less_h}, we obtain
\begin{align*}
    u\circ g(z) - u(z) = \tilde h \circ \tilde g(\tilde z) - \tilde h(\tilde z) -(\tilde h_g\circ \tilde g(\tilde z) -\tilde h_{g}(\tilde z)) = \delta(z),
\end{align*}
for every $z \in \T^2$ and $\tilde z \in \pi^{-1}(z)$.

By Lemma~\ref{lemma:the_pertubed_also_wandering}, $U$ is a wandering disc for $g$, and $\delta$ vanishes outside $U$. Since $g$ satisfies the hypotheses of Theorem~\ref{thm:main-wandering-set},   Lemma~\ref{lemma:our_f_is_weaklytransitive} shows that $\Omega(g)$ is weakly transitive. Finally, Lemma~\ref{lemma:coboundary_obstruction} applied to $g, U$ and $u$ gives $\max \delta = - \min \delta$, contradicting \eqref{eq1:proof_constance}. Therefore $h$ is locally constant at every point of $\mathcal{W}(f)$.
\end{proof}

\begin{proposition}\label{Prop:wandering_lie_interiors}
     The wandering set of $f$ is the union of the interiors of the fibers of $h$, that is,
    \[\mathcal{W}(f) = \bigsqcup_{\theta \in \T}\mathrm{int}(h^{-1}(\theta)).\]
    Moreover, every connected component of $\mathcal{W}(f)$ is contained in a single fiber of $h$.
\end{proposition}

\begin{proof}
    Proposition~\ref{prop:h_is_constant} shows that every point of $\mathcal{W}(f)$ has a neighborhood on which $h$ is constant. Hence every such point belongs to the interior of a fiber. Conversely, let $V = \mathrm{int}(h^{-1}(\theta))$. From the factor identity, $f^n(V) \subset h^{-1}(\theta + n\varrho)$ for every $n \in \Z$. Since $\varrho$ is irrational, $\theta+n\varrho \neq \theta$, for $n \neq 0$, so $f^n(V) \cap V = \emptyset$. Thus $V$ is a wandering set, and $V \subset \mathcal{W}(f)$.

    Finally, $h$ is locally constant on $\mathcal{W}(f)$. A locally constant function is constant on every connected component, proving the last assertion.
    \end{proof}

\begin{corollary}\label{cor:essential}
Every closed curve contained in $\mc W(f)$ is homologous in $\T^2$ to a multiple of the horizontal class $(1,0)$. Consequently, $\mc W(f)$ is either inessential or annular, and $\Omega(f)$ is essential.
\end{corollary}

\begin{proof}
Let $\gamma\colon[0,1]\to\mc W(f)$ be a closed curve, and let $(a,b)\in\Z^2\cong H_1(\T^2,\Z)$ be its homology class. Since $\gamma([0,1])$ is connected, it is contained in a single connected component of $\mc W(f)$, so by Proposition~\ref{Prop:wandering_lie_interiors} there is $\theta\in\T$ with $h\circ\gamma\equiv\theta$. Let $\tilde\gamma\colon[0,1]\to\R^2$ be a lift of $\gamma$, so that $\tilde\gamma(1)=\tilde\gamma(0)+(a,b)$. The continuous function $\tilde h\circ\tilde\gamma$ takes values in the discrete set $\pi^{-1}(\theta)\subset\R$, hence it is constant. On the other hand, since $\tilde h-\pr{2}$ is $\Z^2$-periodic,
\[ \tilde h\bigl(\tilde\gamma(1)\bigr)=\tilde h\bigl(\tilde\gamma(0)\bigr)+b.
\]
Therefore $b=0$. Since $H_1(\mc W(f),\Z)$ is generated by the classes of closed curves in $\mc W(f)$, the image of the homomorphism $H_1(\mc W(f),\Z)\to H_1(\T^2,\Z)$ induced by the inclusion is contained in $\Z\times\{0\}$, and so has rank at most one. That is, $\mc W(f)$ is either inessential or annular and, consequently, $\Omega(f)$ is essential.
\end{proof}

\bibliographystyle{amsplain}
\bibliography{references}

@article{AddasZanataTalGarcia2014,
  AUTHOR = {Addas-Zanata, S. and Tal, F. A. and Garcia, B. A.},
  TITLE = {Dynamics of homeomorphisms of the torus homotopic to {D}ehn twists},
  JOURNAL = {Ergodic Theory Dynam. Systems},
  FJOURNAL = {Ergodic Theory and Dynamical Systems},
  VOLUME = {34},
  YEAR = {2014},
  NUMBER = {2},
  PAGES = {409--422},
  ISSN = {0143-3857},
  MRCLASS = {37E30 (37C25)},
  MRNUMBER = {3186835},
  DOI = {10.1017/etds.2012.156},
  URL = {https://doi.org/10.1017/etds.2012.156},
}

@article{KoropeckiTal2014,
  author  = {Koropecki, A. and Tal, F.},
  title   = {Strictly toral dynamics},
  journal = {Inventiones Mathematicae},
  year    = {2014},
  volume  = {196},
  number  = {2},
  pages   = {339--381},
  doi     = {10.1007/s00222-013-0470-3},
  issn    = {0020-9910}
}

@article{Doeff1997,
  AUTHOR = {Doeff, H. E.},
  TITLE = {Rotation measures for homeomorphisms of the torus homotopic to a {D}ehn twist},
  JOURNAL = {Ergodic Theory Dynam. Systems},
  FJOURNAL = {Ergodic Theory and Dynamical Systems},
  VOLUME = {17},
  YEAR = {1997},
  NUMBER = {3},
  PAGES = {575--591},
  ISSN = {0143-3857},
  MRCLASS = {58F15 (58F25)},
  MRNUMBER = {1457193},
  DOI = {10.1017/S0143385797085015},
  URL = {https://doi.org/10.1017/S0143385797085015},
}

@article{Jager2009,
  AUTHOR = {J{\"a}ger, T.},
  TITLE = {Linearization of conservative toral homeomorphisms},
  JOURNAL = {Invent. Math.},
  FJOURNAL = {Inventiones Mathematicae},
  VOLUME = {176},
  YEAR = {2009},
  NUMBER = {3},
  PAGES = {601--616},
  ISSN = {0020-9910},
  MRCLASS = {37E40 (37C55 37D10)},
  MRNUMBER = {2505595},
  DOI = {10.1007/s00222-008-0171-5},
  URL = {https://doi.org/10.1007/s00222-008-0171-5},
}

@article{JagerTal2017,
  AUTHOR = {J{\"a}ger, T. and Tal, F.A.},
  TITLE = {Irrational rotation factors for conservative torus homeomorphisms},
  JOURNAL = {Ergodic Theory Dynam. Systems},
  FJOURNAL = {Ergodic Theory and Dynamical Systems},
  VOLUME = {37},
  YEAR = {2017},
  NUMBER = {5},
  PAGES = {1537--1546},
  ISSN = {0143-3857},
  MRCLASS = {37E40 (37C25 37E30)},
  MRNUMBER = {3672088},
  DOI = {10.1017/etds.2015.112},
  URL = {https://doi.org/10.1017/etds.2015.112},
}

@article{Kocsard2021,
  AUTHOR = {Kocsard, A.},
  TITLE = {Periodic point free homeomorphisms and irrational rotation factors},
  JOURNAL = {Ergodic Theory Dynam. Systems},
  FJOURNAL = {Ergodic Theory and Dynamical Systems},
  VOLUME = {41},
  YEAR = {2021},
  NUMBER = {10},
  PAGES = {2946--2982},
  ISSN = {0143-3857},
  MRCLASS = {37E30 (37E45)},
  MRNUMBER = {4326553},
  DOI = {10.1017/etds.2020.88},
  URL = {https://doi.org/10.1017/etds.2020.88},
}

@article{PasseggiSambarino2020,
  AUTHOR = {Passeggi, A. and Sambarino, M.},
  TITLE = {Deviations in the {F}ranks--{M}isiurewicz conjecture},
  JOURNAL = {Ergodic Theory Dynam. Systems},
  FJOURNAL = {Ergodic Theory and Dynamical Systems},
  VOLUME = {40},
  YEAR = {2020},
  NUMBER = {9},
  PAGES = {2533--2540},
  ISSN = {0143-3857},
  MRCLASS = {37E45},
  MRNUMBER = {4137255},
  DOI = {10.1017/etds.2019.8},
  URL = {https://doi.org/10.1017/etds.2019.8},
}

@article{KPR18,
  author        = {Kocsard, A. and Pereira-Rodrigues, F.},
  title         = {Rotational deviations and invariant pseudo-foliations for periodic point free torus homeomorphisms},
  journal       = {Mathematische Zeitschrift},
  volume        = {290},
  number        = {3--4},
  pages         = {1223--1247},
  year          = {2018},
  doi           = {10.1007/s00209-018-2060-y},
  eprint        = {1704.04788},
  archivePrefix = {arXiv},
  primaryClass  = {math.DS}
}

@article{MZ89,
  author    = {Misiurewicz, M. and Ziemian, K.},
  title     = {Rotation sets for maps of tori},
  journal   = {Journal of the London Mathematical Society},
  volume    = {40},
  number    = {3},
  pages     = {490--506},
  year      = {1989},
  doi       = {10.1112/jlms/s2-40.3.490}
}

@article {AddasZanataPropVertRotInt,
    AUTHOR = {Addas-Zanata, S.},
     TITLE = {On properties of the vertical rotation interval for twist mappings},
   JOURNAL = {Ergodic Theory Dynam. Systems},
  FJOURNAL = {Ergodic Theory and Dynamical Systems},
    VOLUME = {25},
      YEAR = {2005},
    NUMBER = {3},
     PAGES = {641--660},
      ISSN = {0143-3857,1469-4417},
   MRCLASS = {37E40 (37E45)},
  MRNUMBER = {2142940},
MRREVIEWER = {Paolo\ Perfetti},
       DOI = {10.1017/S014338570400063X},
       URL = {https://doi.org/10.1017/S014338570400063X},
}

@article {AddasZanataPropVertRotIntTwistMapII,
    AUTHOR = {Addas-Zanata, S.},
     TITLE = {On properties of the vertical rotation interval for twist mappings. {II}},
   JOURNAL = {Qual. Theory Dyn. Syst.},
  FJOURNAL = {Qualitative Theory of Dynamical Systems},
    VOLUME = {4},
      YEAR = {2003},
    NUMBER = {2},
     PAGES = {125--137},
      ISSN = {1575-5460,1662-3592},
       DOI = {10.1007/BF02970855},
       URL = {https://doi.org/10.1007/BF02970855}
}

\end{document}